\documentclass{amsart}
\usepackage{amssymb, amsthm, fancyhdr,amsmath,enumerate,graphicx,multicol,psfrag}
\usepackage[inline]{enumitem}
\usepackage{color}
\usepackage{mathrsfs}
\usepackage{bbm}
\usepackage{appendix}
\usepackage[colorlinks=true,linkcolor=blue]{hyperref}
\usepackage{tcolorbox}
\usepackage{mathtools} %
\usepackage{esint}
\mathtoolsset{showonlyrefs=true}  %
\usepackage[left=1.2in,top=1in,right=1.2in,bottom=1in]{geometry}
\usepackage{tikz-cd}
\usepackage{amsaddr}

\DeclareMathOperator{\vol}{vol}

\DeclareMathOperator{\as}{a.s.}

\begin{document}

\def\mf{\mathfrak{M}}
\def\a{\mathcal{A}}
\def\b{\mathcal{B}}
\def\c{\mathcal{C}}
\def\d{\mathscr{D}}
\def\f{\mathcal{F}}
\def\g{\mathcal{G}}
\def\h{\mathcal{H}}
\def\i{\mathcal{I}}
\def\m{\mathcal{M}}
\def\n{\mathcal{N}}
\def\sp{\mathcal{P}}
\def\r{\mathcal{R}}
\def\s{\mathscr{S}}
\def\w{\mathcal{W}}
\def\x{\mathcal{X}}
\def\y{\mathcal{Y}}
\def\z{\mathcal{Z}}
\def\t{\mathcal{T}}

\def\A{\mathbb{A}}
\def\B{\mathbb{B}}
\def\C{\mathbb{C}}
\def\D{\mathbb{D}}
\def\E{\mathbb{E}}
\def\F{\mathbb{F}}
\def\H{\mathbb{H}}
\def\T{\mathbb{T}}
\def\L{\mathcal{L}}
\def\K{\mathbb{K}}
\def\N{\mathbb{N}}
\def\T{\mathbb{T}}
\def\P{\mathbb{P}}
\def\Q{\mathbb{Q}}
\def\R{\mathbb{R}}
\def\W{\mathbb{W}}
\def\Z{\mathbb{Z}}

\def\wstar{\overset{*}{\rightharpoonup}}
\def\nab{\nabla}
\def\dt{\partial_t}
\def\hal{\frac{1}{2}}
\def\ep{\varepsilon}
\def\vchi{\text{\large{$\chi$}}}
\def\ls{\lesssim}
\def\gs{\gtrsim}
\def\p{\partial}

\def\vn{\vol_n}
\def\om{\mu^\ast}

\def\lb{\log_b}

\def\wks{\buildrel\ast\over\rightharpoonup}
\def\wk{\rightharpoonup}

\def\fxx{\F\mspace{-3mu}\jump{\X}}
\def\fl{\F\mspace{-.5mu}(\mspace{-2mu}(\X)\mspace{-2mu})}
\def\afk{\mathbb{A}(\mathbb{K},\mathbb{F})}
\def\cf{\mathbb{C}(\F)}
\def\Re{\text{Re}}
\def\Im{\text{Im}}

\def\af{\mathfrak{A}}
\def\hf{\mathfrak{H}}
\def\ff{\mathfrak{F}}
\def\lf{\mathfrak{L}}
\def\bf{\mathfrak{B}}
\def\mf{\mathfrak{M}}
\def\pf{\mathfrak{P}}
\def\ep{\varepsilon}
\def\qed{{\hfill $\Box$ \bigskip}}
\def\st{\;\vert\;}
\def\vchi{\text{\large{$\chi$}}}

\newlist{anumerate}{enumerate}{1}
\setlist[anumerate,1]{label=(\alph*)}

\newcommand{\abs}[1]{\left\vert#1\right\vert}
\newcommand{\norm}[1]{\left\Vert#1\right\Vert}
\newcommand{\csubset}{\subset\subset}
\newcommand{\ind}[1]{\mathbbm{1}_{#1}}
\newcommand{\qnorm}[1]{\left \vert \mspace{-1.8mu} \left\vert
\mspace{-1.8mu} \left \lvert #1 \right \vert \mspace{-1.8mu} \right\vert
\mspace{-1.8mu} \right\vert}

\newcommand{\br}[1]{\langle #1 \rangle}
\newcommand{\ns}[1]{\norm{#1}^2}
\newcommand{\ip}[1]{\left(#1 \right)}
 
\newcommand{\jump}[1]{\left\llbracket #1 \right\rrbracket }

\newtheorem{thm}{Theorem}[section]
\newtheorem{cor}[thm]{Corollary}
\newtheorem{df}[thm]{Definition}
\newtheorem{assume}[thm]{Assumption}
\newtheorem{prop}[thm]{Proposition}
\newtheorem{rmk}[thm]{Remark}
\newtheorem{lem}[thm]{Lemma}
\newtheorem{ex}[thm]{Example}

\numberwithin{equation}{section}

\title[Concentration inequality for the optimal transport cost]{On concentration of the empirical measure for radial transport costs}

\author[M.\ Larsson, J.\,Park, J.\ Wiesel]{Martin Larsson \textsuperscript{1} , Jonghwa Park \textsuperscript{2}, Johannes Wiesel \textsuperscript{3}}
\address{\textsuperscript{1}
Department of Mathematical Sciences, Carnegie Mellon University, 5000 Forbes Ave, Pittsburgh, PA 15213, USA, larsson@cmu.edu}
\address{\textsuperscript{2}
Department of Mathematical Sciences, Carnegie Mellon University, 5000 Forbes Ave, Pittsburgh, PA 15213, USA, jonghwap@andrew.cmu.edu}
\address{\textsuperscript{3}
Department of Mathematical Sciences, Carnegie Mellon University, 5000 Forbes Ave, Pittsburgh, PA 15213,USA, wiesel@cmu.edu, Corresponding author.}
\subjclass[2023]{}

\begin{abstract}
    Let $\mu$ be a probability measure on $\R^d$ and $\mu_N$ its empirical measure with sample size $N$. We prove a concentration inequality for the optimal transport cost between $\mu$ and $\mu_N$ for radial cost functions with polynomial local growth, that can have superpolynomial global growth. This result generalizes and improves upon estimates of Fournier and Guillin.
    
    The proof combines ideas from empirical process theory with known concentration rates for compactly supported $\mu$. By partitioning $\R^d$ into annuli, we infer a global estimate from local estimates on the annuli and conclude that the global estimate can be expressed as a sum of the local estimate and a mean-deviation probability for which efficient bounds are known.

    \medskip
    
\end{abstract}

\maketitle

\section{Introduction}\label{sec:intro}
Let $d\ge 1$. Denoting the space of probability measures on $\R^d$ by $\sp(\R^d)$, let us consider i.i.d samples $X_1, X_2, \dots$ of $\mu\in \sp(\R^d)$ which are defined on a common probability space $(\Omega, \P)$. We define their empirical measure with sample size $N$ by
\begin{equation} \label{eq_empirical_measure}
\mu_N:= \frac{1}{N} \sum_{i=1}^{N}\delta_{X_i}.
\end{equation}
This paper studies the convergence rate of \textit{the optimal transport cost} between $\mu$ and $\mu_N$. As an important example, we give rates for the cost $\t_p(\mu,\mu_N)$ defined via
\begin{align}\label{eq:for wp}
    \t_p(\mu, \nu)
    :=\inf_{\pi \in \text{Cpl}(\mu, \nu)}\int_{\R^d\times \R^d} \abs{x-y}^p\,d\pi(x,y)\,
\end{align}
for probability measures $\mu,\nu\in \sp(\R^d)$ and some $p>0$. Here $\abs{\cdot}$ refers to the $\ell^2$-norm on $\R^d$ and $\text{Cpl}(\mu, \nu)$ is the set of couplings of $\mu$ and $\nu$. The cost $\t_p$ is related to the $p$-Wasserstein distance $\w_p$ via $\t_p=\w_p^p$. Estimating the convergence rates of $\t_p(\mu, \mu_N)$ is a classical problem in many areas of mathematics, such as PDE and probability theory, and has been studied extensively in the literature. For example, moment bounds on $\t_p(\mu, \mu_N)$ can be found in \cite{HOROWITZ1994261, dereich2013constructive, mischler2013kac, boissard2014mean, fournier2015rate, 10.3150/18-BEJ1065, singh2019minimax, lei2020convergence}. Non-asymptotic concentration inequalities (or deviation inequalities) for $\t_p(\mu, \mu_N)$ were established by \cite{bolley2007quantitative, gozlan2007large, Boissard2011SimpleBF, fournier2015rate, dedecker2015deviation, lei2020convergence}.

To motivate our work, let us  consider a probability measure $\mu\in \sp(\R^d)$ that is supported on a closed ball of radius $1/2$ centered at $0$. It follows e.g.\ from \cite[Proposition 10]{fournier2015rate} that for all $N\in \N$ and $x>0$,
\begin{align}\label{eq:intro cpt}
    \P(\t_p(\mu, \mu_N)>x)\le Ce^{-cN\varphi(x)}\ind{\{x\le 1\}},
\end{align}
where the rate function $\varphi\colon [0,\infty)\to [0,\infty)$ is defined as
\begin{equation} \label{eq_rate_function}
    \varphi(x)
    :=\begin{dcases}
    x^2 & \text{ if } p>d/2\\
    (x/\log(2+1/x))^2 & \text{ if } p=d/2\\
    x^{d/p} & \text{ if } p\in (0,d/2).
    \end{dcases}
\end{equation}
In the above, the constants $c>0$ and $C>0$ only depend on $d$ and $p$, but not on $\mu$. 

Let us rewrite the deviation inequality \eqref{eq:intro cpt} in an alternative manner. Since $\mu$ is compactly supported on a ball of radius $1/2$, Hoeffding's lemma (see \cite[Theorem 2.1]{devroye2001combinatorial}) yields
\begin{align}
    \P(\abs{M_p(\mu)-M_p(\mu_N)}>x)
    \le 2e^{-2^{2p+1} N x^2}\,
\end{align}
where $M_p(\mu_N):=(1/N)\sum_{j=1}^{N}\abs{X_j}^p$ and $M_p(\mu):=\E[\abs{X_1}^p]$. Note that $2e^{-2^{2p+1}N x^2}$ can be absorbed into $Ce^{-cN\varphi(x)}$ if $x\in [0,1]$, possibly with different constants $c$ and $C$. Thus, up to a change of constants, the deviation inequality \eqref{eq:intro cpt} is equivalent to the estimate
\begin{align}\label{eq:intro grl}
    \P(\t_p(\mu, \mu_N)>x)\le Ce^{-cN\varphi(x)}\ind{\{x\le 1\}}
    +\P(\abs{M_p(\mu)-M_p(\mu_N)}>x).
\end{align}

Although the results in \cite{fournier2015rate} strongly suggest that \eqref{eq:intro grl} should remain true for a very large class of (not necessarily compactly supported) laws $\mu\in \sp(\R^d)$, the authors only prove weaker estimates. Indeed, \cite[Theorem 2]{fournier2015rate} is of the form
\begin{align}
    \P(\t_p(\mu, \mu_N)>x)
    \le C e^{-cN\varphi(x)}\ind{\{x\le 1\}}+b(N,x)\,
\end{align}
for some rate function $b(N,x)$. Compared to well-known upper bounds of $\P(\abs{M_p(\mu)-M_p(\mu_N)}>x)$, the rate function $b(N,x)$ yields less stringent estimates. In conclusion, concentration estimates provided in \cite{fournier2015rate} are generally weaker than \eqref{eq:intro grl}.

The aim of this paper is to prove the deviation inequality \eqref{eq:intro grl} in full generality. This will follow from a more general concentration inequality for cost functions that are \textit{locally polynomial of order $p$}. Informally it can be stated as follows. Consider a measurable cost function $f \colon [0,\infty) \to [0,\infty)$ such that $\sup_{0<r\le R} f(r)/r^p < \infty$ for all $R > 0$ and define
\begin{align}\label{eq:for d}
    \d_f(\mu, \nu)
    :=\inf_{\pi \in \text{Cpl}(\mu, \nu)}\int_{\R^d\times \R^d} f(\abs{x-y})\,d\pi(x,y) \text{ for } \mu, \nu\in \sp(\R^d).
\end{align}
Under mild moment conditions on $\mu$, we show that the concentration estimate \eqref{eq:intro grl} generalizes to $\d_f(\mu, \mu_N)$.

The contributions of this paper are threefold. First, we improve the deviation inequalities for $\t_p$ that were established in \cite{fournier2015rate}: on the one hand we relax the assumptions on $\mu$, thus obtaining \eqref{eq:intro grl} for a larger class of probability measures. On the other hand, when restricted to the same class of $\mu$ as in \cite{fournier2015rate}, we provide strictly stronger estimates. We mention one example here and refer to Section~\ref{sec:examples} for a detailed comparison with \cite{fournier2015rate}. If $\E[e^{a\abs{X_1}^{\beta}}]<\infty$ for some $a>0$ and $\beta\in (0,p)$, we prove that
\begin{align}
    \P(\t_p(\mu, \mu_N)>x)
    \le C(e^{-cN\varphi(x)}\ind{\{x\le 1\}}+e^{-c(Nx)^{\beta/p}}).
\end{align}
This implies the bound in \cite[Remark 3]{fournier2015rate}, namely
\begin{align}\label{eq:fg_rate_ex}
    \P(\t_p(\mu, \mu_N)>x)
    \le C\bigg(e^{-cN\varphi(x)}\ind{\{x\le 1\}}
    +e^{-cNx^2(\log(N+1))^{1-2p/\beta}}
    +e^{-c(Nx)^{\beta/p}}\bigg).
\end{align}
Our estimate improves upon \eqref{eq:fg_rate_ex} for small values of $x$. To see this, let $p>d/2$ and set $x_N=\sqrt{\log(N+1)/N}$. Our estimate then implies $\P(\t_p(\mu, \mu_N)>x_N)\to 0$ as $N\to \infty$ which does not follow from \eqref{eq:fg_rate_ex}.

Our second contribution is of methodological nature: we derive new concentration inequalities based on concepts from \textit{empirical process theory} and combine these with optimal transport techniques to derive sharper concentration inequalities. This method is \textit{robust} in the sense that it offers general estimates that apply to wide range of laws $\mu$, and can be seen as a generalization of techniques used in \cite{fournier2015rate}, where sophisticated control of binomial random variables played an essential role. In particular, the proof of \cite[Theorem 2]{fournier2015rate} requires three different techniques to control binomial random variables, each corresponding to three different assumptions imposed on $\mu$. In contrast, our paper is based on universal tools that control the \textit{uniform deviation of self-normalized empirical processes} and goes back to ideas formulated in \cite{vapnik1974theory, anthony1993result, bartlett1999inequality} and \cite[Exercise 3.3, Exercise 3.4]{devroye2001combinatorial}. Empirical process theory was also used in \cite{manole2021sharp}, albeit in a different way.

Finally, our paper contains sharp moment bounds for some non-Wasserstein costs $\d_f\neq \t_p$. To give an example, let us define the transport cost $\mathcal{E}_{p,a}$ for $a,p>0$ via
\begin{align}\label{eq:for exp}
    \mathcal{E}_{p,a}(\mu, \nu)
    :=\inf_{\pi \in \text{Cpl}(\mu, \nu)}\int_{\R^d\times \R^d} \left(e^{a\abs{x-y}^p}-1\right)\,d\pi(x,y)\text{ for } \mu, \nu\in \sp(\R^d).
\end{align}
Assume that $\E[e^{b\abs{X_1}^p}]<\infty$ for $b>2^{2p+1}a$ if $p\ge 1$, and for $b>2^{p+2}a$ if $p\in (0,1)$. We then prove the moment bounds
\begin{align}
    \E[\mathcal{E}_{p,a}(\mu, \mu_N)]
    \le C
    \begin{dcases}
     N^{-1/2} & \text{ if } p>d/2\,\\
    \log(N+1)N^{-1/2} & \text{ if } p=d/2\,\\
    N^{-p/d} & \text{ if } p\in (0,d/2),
    \end{dcases}
\end{align}
see Examples~\ref{ex:mom expdhp} and~\ref{ex:mom expdlp}. The sharpness of these bounds is discussed at the end of Section~\ref{sec:examples}.

\subsection{Related work}
In this section we provide a short overview of existing work on Wasserstein rates between the true and empirical measure, departing from \cite{fournier2015rate}.

\subsubsection*{Moment bounds}
Moment bounds for $\mathcal{W}_p(\mu, \mu_N)$ are extended from Euclidean spaces to arbitrary compact metric spaces in \cite{10.3150/18-BEJ1065}, where it is shown that the rate of convergence depends on the so-called intrinsic dimension of the measure $\mu$. An extension of moment bounds to unbounded Banach spaces such as separable Hilbert spaces is carried out in \cite{lei2020convergence}. Using minimax theory, \cite{singh2019minimax} gives a different proof of the bounds of \cite{fournier2015rate}.

There has been some recent effort in making the generic constants appearing in the moment bounds explicit, see \cite{kloeckner2020empirical} and \cite{fournier2022convergence}. \cite{kloeckner2020empirical} derives explicit bounds for $1$-Wasserstein distances when the $\mu$ is supported on $[0,1]^d$. \cite{fournier2022convergence} covers the general case, and shows that constants can be chosen in such a way that they do not explode for $d\to \infty$.

\subsubsection*{Concentration estimates} Concerning concentration inequalities, fewer results are known. They are mostly derived from estimates of the deviation of $\w_p(\mu, \mu_N)$ around its mean, i.e.\ $$\P(\abs{\w_p(\mu, \mu_N)-\E[\w_p(\mu, \mu_N)]}>x).$$ For measures on Polish spaces, \cite{dedecker2015deviation} deduce mean-concentration inequalities for $1$-Wasserstein distances from Lipschitz duality. \cite{lei2020convergence} obtains explicit inequalities for measures satisfying a Bernstein-type tail condition. All of these results impose certain quantitative conditions on the moments of $\mu$, which are more restrictive than the conditions required in this paper.

\subsubsection*{Empirical process theory} The recent work \cite{manole2021sharp} incorporates empirical process theory into the analysis of plug-in estimators for $\d_f(\mu, \nu)$, where $\mu, \nu\in \sp(\R^d)$ are two probability distributions of interest. In practice, $\mu$ and $\nu$ are often unknown and need to be approximated by their empirical counterparts $\mu_N$ and $\nu_N$. Motivated by this, \cite{manole2021sharp} study the expected error between the optimal cost $\d_f(\mu, \nu)$ and its plug-in estimator $\d_f(\mu_N, \nu_N)$. Their methodology relies on empirical process theory for the \textit{uniform deviation}, developed in \cite{von2004distance}. For the special case $\d_f=\w^p_p$, they establish faster convergences rates compared to the ones obtained by applying the triangle inequality to the rate of \cite{fournier2015rate}.

However, their results are based on certain regularity conditions on the cost function $f$ as well as global growth conditions on the derivative $f'$. These conditions are stronger than ours. In particular, they exclude the exponential cost $\mathcal{E}_{p,a}$. Furthermore,  \cite{manole2021sharp} makes stronger assumptions on the measures $\mu$ and $\nu$, which are required to satisfy a sub-Weibull condition (implying finite exponential moments) and have densities with polynomially growing logarithmic gradients.

\subsection{Organization of the paper}
The rest of the paper is organized as follows. We end this introduction by establishing notation in Subsection~\ref{S_notation}. In Section~\ref{sec:main results}, we discuss our main assumptions and present our main result, Theorem~\ref{thm:goal}. In Section~\ref{sec:examples}, we elaborate on various examples. In particular we compute concentration inequalities for $\t_p$ and compare them with existing results. Estimates for $\mathcal{E}_{p,a}$ are also presented in this section. We test our main result numerically in Section \ref{sec:numericaltest} for various distributions $\mu$. Section~\ref{sec:proof} is devoted to proofs.

\subsection{Notation} \label{S_notation}
We denote the set of non--negative integers by $\N$. As usual, we write $x_+:=\max(x,0)$ for any real number $x$. All random objects are defined on a sample space $(\Omega,\P)$ with expectation operator $\E$. The empirical measure $\mu_N$ of a probability measure $\mu\in \sp(\R^d)$ based on a sample size $N\ge 1$ is defined in \eqref{eq_empirical_measure}. For technical reasons we also define $\mu_N:=\delta_0$ if $N=0$, where $\delta_0$ denotes the Dirac measure at $0$. We also recall that $\t_p$, $\d_f$ and $\mathcal{E}_{p,a}$ are defined in \eqref{eq:for wp}, \eqref{eq:for d} and \eqref{eq:for exp}, respectively.

Let $\mu\in \sp(\R^d)$. For $q>0$ and a measurable function $h \colon [0,\infty)\to [0,\infty)$, we denote its moments by
\begin{align}
    M_q(\mu):=\int_{\R^d} \abs{y}^q\,d\mu(y),\quad
    M_q(\mu; h):=\int_{\R^d}h(\abs{y})^q\,d\mu(y).
\end{align}
In particular, $M_1(\mu; h)$ is the integral of $h(\abs{\cdot})$ with respect to $\mu$.

Next, the rate function $\varphi$ is defined in \eqref{eq_rate_function}. Note that $\varphi$ depends on the dimension $d$ and a growth parameter $p>0$, even though this is not explicit in the notation. Recall from \eqref{eq:intro cpt} that $\varphi$ controls the rate of concentration for compactly supported measures. For general probability measures $\mu\in \sp(\R^d)$, we will work with the following variation of $\varphi$. We set
\begin{equation} \label{eq_phi_eta}
    \varphi_{\eta}(x)
    :=\begin{dcases} x^{1/\eta} &\text{ if } p\neq d/2,\,\\
    x^{1/\eta}/(\log(2+1/x))^2 &\text{ if } p=d/2\,
    \end{dcases}
\end{equation}
for $\eta\in (0, \frac{1}{2} \wedge \frac{p}{d}]$. Note that $\varphi_{\eta}=\varphi$ when $\eta=\frac{1}{2} \wedge \frac{p}{d}$.

\section{Main result}\label{sec:main results}

We now introduce our main assumptions and present our main result. Throughout this paper, we focus on the optimal transport cost $\d_f$ for a cost function $f$ and a measure $\mu \in \sp(\R^d)$. We now detail the assumptions on $f$ and $\mu$, which will be central for our main result.

\begin{assume}\label{ass:cst} Fix a dimension $d \ge 1$ and a growth parameter $p \in (0,\infty)$. The cost function $f\colon [0,\infty)\to [0,\infty)$ and the measure $\mu\in \sp(\R^d)$ satisfy the following hypotheses:
\begin{anumerate}
    \item(Continuity)\label{ass:cont} $f$ is lower semicontinuous.
    \item(Growth)\label{ass:loc g} $f$ satisfies
    \begin{equation} \label{eq_local_growth}
    \sup_{0 < r \le R} \frac{f(r)}{r^p} < \infty \text{ for all } R \in (0,\infty).
    \end{equation}
    We then fix a measurable function $g \colon [0,\infty)\to [0,\infty)$ and a nondecreasing function $G \colon [0,\infty)\to [0,\infty)$ such that
    \begin{align}
        f(R\abs{x})\le g(R)\abs{x}^p \text{ for every } x \in \R^d \text{ with } \abs{x}\le 1
    \end{align}
    and
    \begin{align}
        f(\abs{x-y})\le G(\abs{x})+G(\abs{y}) \text{ for every } x,y\in \R^d.
    \end{align}
    Such functions can always be found when \eqref{eq_local_growth} is satisfied, for example $g(R) = R^p \sup_{0 < r \le R} f(r) / r^p$ and $G(R) = \sup_{r \le 2R} f(r)$, although other choices are sometimes preferable.

    \item(Moment conditions)\label{ass:mom} 
    \begin{enumerate}[label=(c\arabic*)]
        \item $M_{p}(\mu)<\infty$ and $M_{\gamma}(\mu; G)<\infty$ for some $\gamma\in (1,\infty)$.
        \item There exists a nondecreasing function $S\colon [0,\infty)\to (0,\infty)$ such that $M_1(\mu; S)<\infty$ and
        \begin{align}
            \mathcal{K}_{g}:=g(2)+\sum_{k=1}^{\infty}\frac{2^{kc_0}g(2^{k+1})}{S(2^{k-1})^{1-\eta}}<\infty,\quad
            \mathcal{K}_{G}:=G(1)+\sum_{k=1}^{\infty}\frac{2^{kc_0}G(2^k)}{S(2^{k-1})^{1-\eta}}<\infty\,
        \end{align}
        for some $c_0>0$ and some $\eta\in (0, \frac{1}{2} \wedge \frac{p}{d}]$, i.e.\ $\eta\in (0,1/2]$ if $p\ge d/2$ and $\eta\in (0,p/d]$ if $p\in (0,d/2)$.
    \end{enumerate}
\end{anumerate}
\end{assume}

The role of the continuity condition~$\ref{ass:cont}$ is to ensure that an optimal transport plan for $\d_f$ exists; see \cite[Theorem 4.1]{villani2009optimal} for details. The growth condition~$\ref{ass:loc g}$ states that $f(r)$ is locally bounded for large $r$ and decays to zero at least as fast as $r^p$ for small  $r$.

There are many interesting functions $f$ that satisfy conditions~$\ref{ass:cont}$ and~$\ref{ass:loc g}$. One obvious example is $f(r)=r^p$. In this case, the optimal transport cost is $\d_f=\t_p$ and we may choose $g=f$ and $G(r)=c_p r^p$ where $c_p:=2^{p-1}$ if $p\ge 1$ and $c_p:=1$ if $p\in (0,1)$. Another interesting example is the exponential function $f(r)=e^{ar^p}-1$ for some $a>0$. In this case we have $\d_f=\mathcal{E}_{p,a}.$ Using the fact that $y\mapsto \frac{e^{a y}-1}{y}$ is increasing, we can take $g=f$. Also, thanks to Jensen's inequality, we can choose $G(r)=\frac{1}{2}(e^{2c_pa r^p}-1)$, where $c_p$ is as above.

One intuitive way to understand moment condition~$\ref{ass:mom}$ is to view $\gamma$ and $\eta$ as \textit{indicators of how stringent the assumptions on $\mu$ are}. More precisely, larger values of $\gamma$ and $\eta$ indicate stronger assumptions on $\mu$ and vice versa. This is obvious for $\gamma$. For $\eta$, the argument is as follows: to satisfy~$\ref{ass:mom}$ we will choose the largest possible function $S$ that is $\mu$-integrable. As $S$ gets larger, we are allowed to choose larger $\eta$ that makes both $\mathcal{K}_g$ and $\mathcal{K}_{G}$ finite. Hence, a larger value of $\eta$ indicates stronger assumptions on moments of $\mu$. For example, take $f(r)=r^p$ and consider $\d_f=\t_p$. Assume that $M_q(\mu)<\infty$ for $q>p$ and set $S(r)=1\vee r^{q}$. Then
\begin{align}\label{eq:kforw}
    \mathcal{K}_g=2^p+\sum_{k=1}^{\infty}\frac{2^{kc_0}2^{p(k+1)}}{(1\vee 2^{q(k-1)})^{1-\eta}},\quad
    \mathcal{K}_G=1+\sum_{k=1}^{\infty}\frac{2^{kc_0}c_p2^{kp}}{(1\vee 2^{q(k-1)})^{1-\eta}}\,
\end{align}
and we compute that $\mathcal{K}_g<\infty$ and $\mathcal{K}_G<\infty$ if and only if $q>p/(1-\eta)$. In turn, larger $\eta$ means $M_q(\mu)<\infty$ for larger $q$ and vice versa.

\begin{rmk}
It is worth mentioning that the condition $M_{\gamma}(\mu; G)<\infty$ for $\gamma>1$ implies that $\d_f(\mu, \mu_N)<\infty \,\,\as$ Furthermore, the moment condition $M_p(\mu)<\infty$ could be relaxed to requiring $M_r(\mu) < \infty$ for some $r \in (0,\infty)$ allowed to be strictly smaller than $p$.
\end{rmk}

Let us now state our main result.

\begin{thm}\label{thm:goal} Let us assume that Assumption~\ref{ass:cst} is satisfied. Then the following estimates hold.
\begin{anumerate}
    \item Let $\gamma>2$. Then there exist positive constants $c, C, A_0>0$ such that for all $N\in \N$ and $x>0$,
    \[
        \P(\d_f(\mu, \mu_N)>Fx)
        \le Ce^{-cN\varphi_{\eta}(x)}\ind{\{x\le A_0\}}
        +\P\left(M_1(\mu_N; G)-M_1(\mu; G)>x\right)
    \]
    and
    \[
        \P(\d_f(\mu, \mu_N)>F_Nx)
        \le Ce^{-cN\varphi_{\eta}(x)}\ind{\{x\le A_0\}}
        +\P\left(M_1(\mu; G)-M_1(\mu_N; G)>x\right),
    \]
    where
    \begin{align}
        F&:=(1\vee M_1(\mu; S))^{1-\eta}
        +(1\vee M_{p}(\mu))^{1/2-1/\gamma}(M_{\gamma}(\mu; G))^{1/\gamma},\\
        F_N&:=(1\vee M_1(\mu_N; S))^{1-\eta}
        +(1\vee M_{p}(\mu_N))^{1/2-1/\gamma}(M_{\gamma}(\mu_N; G))^{1/\gamma}.
    \end{align}
    Here, the constant $c$ depends only on $d,p, \gamma, \eta, \mathcal{K}_g, \mathcal{K}_G$, the constant $C$ depends only on $d, p, \gamma ,\delta_0$ and the constant $A_0$ depends only on $p,\gamma, \mathcal{K}_g, \mathcal{K}_G$.
    \item Let $\gamma\in (1,2]$ and fix $\ep\in (0, 1-1/\gamma)$. For some positive constants $c, C, A_0>0$, one has for all $N\in \N$ and $x>0$,
    \begin{align}
        &\P(\d_f(\mu, \mu_N)>Fx)\\
        &\quad \le C\left(e^{-cN\varphi_{\eta}(x)}+e^{-c N^{2(1-1/\gamma-\ep)}x^2}\right)\ind{\{x\le A_0\}}
        +\P\left(M_1(\mu_N; G)-M_1(\mu; G)>x\right)
    \end{align}
    and
    \begin{align}
        &\P(\d_f(\mu, \mu_N)>F_Nx)\\
        &\quad \le C\left(e^{-cN\varphi_{\eta}(x)}+e^{-c N^{2(1-1/\gamma-\ep)}x^2}\right)\ind{\{x\le A_0\}}
        +\P\left(M_1(\mu; G)-M_1(\mu_N; G)>x\right),
    \end{align}
    where
    \begin{align}
        F&:=(1\vee M_1(\mu; S))^{1-\eta}
        +(1\vee M_{p}(\mu))^{\ep}(M_{\gamma}(\mu; G))^{1/\gamma},\\
        F_N&:=(1\vee M_1(\mu_N; S))^{1-\eta}
        +(1\vee M_{p}(\mu_N))^{\ep}(M_{\gamma}(\mu_N; G))^{1/\gamma}.
    \end{align}
    Here, the constants $c, C$ and $A_0$ depend only on both $\ep$ and the same set of parameters as in the previous case.
\end{anumerate}
\end{thm}
Naturally, our estimates improve with the magnitude of $\gamma$ and $\eta$. Indeed, if $x\le A_0$ and $\eta_1>\eta_2$, then $\varphi_{\eta_1}(x)\ge c\varphi_{\eta_2}(x)$ for some generic constant $c>0$. This observation is consistent with our previous interpretation that larger values of $\gamma$ and $\eta$ are indicating stronger moment conditions on $\mu$. 

The above estimates are stated in terms of the deviation between the true mean $M_1(\mu; G)$ and the empirical mean $M_1(\mu_N; G)$. Deviation estimates for these quantities are well-studied, see e.g.\ Proposition~\ref{prop:dvp} below for a summary.

\begin{rmk}
In a statistical inference framework, the theoretical moments of the population in the inequalities above are typically unknown and have to be replaced by their empirical counterparts. In particular, $F$ is approximated by $F_N$. Since the constants $c, C$ and $A_0$ do not depend on the moments of $\mu$, the second estimates for each case in Theorem~\ref{thm:goal} above may be useful in this case.
\end{rmk}

\section{Examples}\label{sec:examples}
In order to emphasize the flexibility of the concentration bounds stated in Theorem~\ref{thm:goal}, we now explicitly compute them for various functions $f$. To achieve this we need to estimate the deviation of the difference between $M_1(\mu; G)$ and $M_1(\mu_N; G)$. 

Recall that $X_1, \dots, X_N$ are i.i.d.~samples of $\mu$ and set $Y_i:=G(\abs{X_i})$. Then $Y_i$ are i.i.d.~random variables with mean $\E[Y_1]=M_1(\mu; G)$ and $M_1(\mu_N; G)$ is the empirical mean of $Y_i$. In consequence we need to estimate the difference between the empirical mean and the true mean of $Y_i$. This problem is well-studied in the literature, and some well-known results can be found in \cite{fournier2015rate}. We complement these in the next proposition below.

\begin{prop}\label{prop:dvp} Let $Y_1,Y_2,\dots$ be i.i.d.~random variables and define $\overline{Y}_N:=(1/N)\sum_{i=1}^N Y_i$.
\begin{anumerate}
    \item\label{prop:eg} Suppose $\E[e^{a\abs{Y_1}^{\beta}}]<\infty$ for $a>0$ and $\beta\ge 1$. Then for all $N\in \N$ and $x>0$,
    \begin{align}
        \P\left(\abs{\overline{Y}_N-\E[Y_1]}>x\right)
        \le C\left(e^{-cNx^2}\ind{\{x\le 1\}}
        +e^{-c N x^{\beta}}\ind{\{x>1\}}\right)\,
    \end{align}
    for some positive constants $c$ and $C$ that depend only on $a, \beta, \E[Y_1]$ and $\E[e^{a\abs{Y_1}^{\beta}}]$.
    \item\label{prop:el} Suppose $\E[e^{a\abs{Y_1}^{\beta}}]<\infty$ for $a>0$ and $\beta\in (0,1)$. Then for all $N\in \N$ and $x>0$,
    \begin{align}
        \P\left(\abs{\overline{Y}_N-\E[Y_1]}>x\right)
        \le C\left(e^{-cNx^2}
        +e^{-c (N x)^{\beta}}\right)\,
    \end{align}
    for some positive constants $c$ and $C$ that depend only on $a, \beta, \E[Y_1]$ and $\E[e^{a\abs{Y_1}^{\beta}}]$.
    \item\label{prop:mg} Suppose $\E[\abs{Y_1}^t]<\infty$ for $t>2$. Then for all $N\in \N$ and $x>0$,
    \begin{align}
        \P\left(\abs{\overline{Y}_N-\E[Y_1]}>x\right)
        \le e^{-cNx^2}
        +CN(Nx)^{-t}\,
    \end{align}
    for some positive constants $c$ and $C$ that depend only on $t$ and $\E[\abs{Y_1}^t]$.
    \item\label{prop:ml} Suppose $\E[\abs{Y_1}^t]<\infty$ for $t\in [1, 2]$. Then for all $N\in \N$ and $x>0$,
    \begin{align}
        \P\left(\abs{\overline{Y}_N-\E[Y_1]}>x\right)
        \le CN(Nx)^{-t}\,
    \end{align}
    for some positive constant $C$ that depends only on $t$ and $\E[\abs{Y_1}^t]$.
\end{anumerate}
\end{prop}
\begin{proof}
The first estimate in Proposition~\ref{prop:dvp} can be deduced from the transportation inequality, see \cite{djellout2004transportation}, \cite{gozlan2006integral} and \cite{alma991006972269704436} for more details. The estimate~$\ref{prop:el}$ follows from \cite[Formula (1.4)]{merlevede2011bernstein}, which is based on \cite[Corollary 5.1]{borovkov2000estimates}. The argument is as follows: \cite[Formula (1.4)]{merlevede2011bernstein} gives the upper bound $e^{-cNx^2}+Ne^{-c(Nx)^{\beta}}$. When $Nx^2<A$ and $A>0$ is large enough, the estimate in~$\ref{prop:el}$ becomes greater than $1$. If $Nx^2\ge A$, it is easy to check that $e^{-cNx^2}+Ne^{-c(Nx)^{\beta}}$ is dominated by $C(e^{-cNx^2}+e^{-c(Nx)^{\beta}})$ possibly with a different constant $c$. For~$\ref{prop:mg},$ see \cite[Corollary 4]{fuk1971probability}. The last result is stated in \cite[Section 5]{fuk1971probability} and \cite{vonBahrBengt1965Iftr}.
\end{proof}

\subsection{Estimates for compactly supported measures}
Let us begin our discussion by considering a probability measure $\mu$ supported on $\overline{B_R(0)}$, the closed ball with radius $R>0$ centered at $0$. Since $G$ is non-decreasing and thus $\mu$-integrable, $\gamma$ can be chosen to be any positive number. Let us choose $\gamma>2$. Next we choose $S$ sufficiently large, so that $\mathcal{K}_g$ and $\mathcal{K}_G$ are both finite for $\eta=\frac{1}{2} \wedge \frac{p}{d}$ and $c_0=1$. Hoeffding's lemma (see \cite[Theorem 2.1]{devroye2001combinatorial}) gives
\begin{align}
    \P\left(M_1(\mu_N; G)-M_1(\mu; G)>x\right)
    \le e^{-cNx^2}\ind{\{x\le A_0\}}\,
\end{align}
for some positive constants $c$ and $A_0$ that depend on $G$ and $R$ only. Combining this result with Theorem~\ref{thm:goal}, we obtain the following concentration inequality.
\begin{ex}[$\d_f$ assuming compact support]\label{ex:cpt} Suppose $\mu$ is supported on $\overline{B_{R}(0)}$. Then for all $N\in \N$ and $x>0$,
\begin{align}
    \P(\d_f(\mu, \mu_N)>x)
    \le Ce^{-cN\varphi(x)}\ind{\{x\le A_0\}}\,
\end{align}
for some positive constants $c$, $C$ and $A_0$ that depend only on $d, p, g, G, R$. Taking $\d_f = \t_p$ we recover the Fournier--Guillin bound recalled in Lemma~\ref{lem:cmpbd} below, which is used in the proof of Theorem~\ref{thm:goal}. We are thus reassured that the proof of the theorem does not lead to any loss in sharpness in the case of compactly supported distributions.
\end{ex}

\subsection{Estimates for \texorpdfstring{$\t_p$}{TEXT}}
Next we consider $f(r)=r^p$ for $p>0$. As discussed earlier, we take $g=f$ and $G(r)=c_pr^p$, where $c_p=2^{p-1}$ if $p\ge 1$ and $c_p=1$ if $p\in (0,1)$. Let us assume that $M_q(\mu)<\infty$ for some $q>p$, and set $\gamma=q/p$, $S(r)=1\vee r^q$. As computed in \eqref{eq:kforw}, this forces $\eta\in (0,1-p/q)\cap(0, \frac{1}{2} \wedge \frac{p}{d}]$. Note that $c_0$ can be chosen to be any positive number less than $q(1-\eta)-p$. We now choose $\eta$ as large as possible in order to get the best possible estimate from Theorem~\ref{thm:goal}. Comparing $1-p/q$ with $\frac{1}{2} \wedge \frac{p}{d}$, we obtain the following:
\begin{enumerate}
    \item If $p\ge d/2$ and $q>2p$, then $\gamma>2$ and $\eta=1/2$.
    \item If $p\ge d/2$ and $q\in (p,2p]$, then $\gamma\le 2$ and $\eta=1-p/q-\ep$ for arbitrarily small $\ep\in (0,1-p/q)$.
    \item If $p\in (0,d/2)$ and $q>2p$, then $\gamma>2$ and $\eta=p/d$.
    \item If $p\in (0,d/2)$ and $q\in(dp/(d-p),2p]$, then $\gamma\le 2$ and $\eta=p/d$.
    \item If $p\in (0,d/2)$ and $q\in (p, dp/(d-p)]$, then $\gamma\le 2$ and $\eta=1-p/q-\ep$ for arbitrarily small $\ep\in (0,1-p/q)$.
\end{enumerate}
Combining these results with Theorem~\ref{thm:goal} and Proposition~\ref{prop:dvp}$\ref{prop:mg}$--$\ref{prop:ml}$, we establish the following non-asymptotic concentration inequalities for $\t_p$:

\begin{ex}[$\t_p$ assuming finite $q$th moment]\label{ex:wsm} Let us assume $M_q(\mu)<\infty$ for some $q>p$.
\begin{anumerate}
    \item\label{ex:wsm hq} If $q>2p$, then for all $N\in \N$ and $x>0$,
    \begin{align}
        \P(\t_p(\mu, \mu_N)>x)
        \le C\bigg(e^{-cN\varphi(x)}\ind{\{x\le 1\}}
        +N(Nx)^{-q/p}\bigg).
    \end{align}
    \item If $p> d/2$ and $q\in (p,2p]$, then for all $N\in \N$, $\ep\in (0,1-p/q)$ and $x>0$,
    \begin{align}
        \P(\t_p(\mu, \mu_N)>x)
        \le C\bigg(e^{-cN^{2(1-p/q-\ep)}x^2}\ind{\{x\le 1\}}
        +N(Nx)^{-q/p}\bigg).
    \end{align}
    \item If $p=d/2$ and $q\in (p,2p]$, then for all $N\in \N$, $\ep\in (0,1-p/q)$ and $x>0$,
    \begin{align}
        &\P(\t_p(\mu, \mu_N)>x)\\
        &\le C\bigg(e^{-cNx^{1/(1-p/q-\ep)}/(\log(2+1/x))^2}\ind{\{x\le 1\}}
        +e^{-cN^{2(1-p/q-\ep)}x^2}\ind{\{x\le 1\}}
        +N(Nx)^{-q/p}\bigg).
    \end{align}
    \item If $p\in (0,d/2)$ and $q\in (dp/(d-p), 2p]$, then for all $N\in \N$, $\ep\in (0,1-p/q)$ and $x>0$,
    \begin{align}
        \P(\t_p(\mu, \mu_N)>x)
        \le C\bigg(e^{-cNx^{d/p}}\ind{\{x\le 1\}}
        +e^{-cN^{2(1-p/q-\ep)}x^2}\ind{\{x\le 1\}}
        +N(Nx)^{-q/p}\bigg).
    \end{align}
    \item If $p\in (0,d/2)$ and $q\in (p, dp/(d-p)]$, then for all $N\in \N$, $\ep\in (0,1-p/q)$ and $x>0$,
    \begin{align}
        \P(\t_p(\mu, \mu_N)>x)
        \le C\bigg(e^{-cN^{2(1-p/q-\ep)}x^2}\ind{\{x\le 1\}}
        +N(Nx)^{-q/p}\bigg).
    \end{align}
\end{anumerate}
Here, the constants $c>0$ and $C>0$ depend only on $d,p,q, M_q(\mu)$ and if applicable, also on $\ep$.
\end{ex}

Assuming that the law $\mu$ has an exponential moment, these estimates can be improved by use of Proposition~\ref{prop:dvp}$\ref{prop:eg}$--$\ref{prop:el}$.

\begin{ex}[$\t_p$ assuming finite exponential moment]\label{ex:wsem} Let us assume $E:=\int_{\R^d}e^{a\abs{y}^{\beta}}\,d\mu(y)<\infty$ for some $a, \beta>0$.
\begin{anumerate}
    \item\label{ex:wsem hb} If $\beta\ge p$, then for all $N\in \N$ and $x>0$,
    \begin{align}
        \P(\t_p(\mu, \mu_N)>x)
        \le C\left(e^{-cN\varphi(x)}\ind{\{x\le 1\}}
        +e^{-cNx^{\beta/p}}\ind{\{x\ge 1\}}\right).
    \end{align}
    \item\label{ex:wsem lb} If $\beta\in (0,p)$, then for all $N\in \N$ and $x>0$,
    \begin{align}
        \P(\t_p(\mu, \mu_N)>x)
        \le C\left(e^{-cN\varphi(x)}\ind{\{x\le 1\}}
        +e^{-c(Nx)^{\beta/p}}\right).
    \end{align}
\end{anumerate}
Here, the constants $c>0$ and $C>0$ depend only on $d,p,a, \beta, E$.
\end{ex}
\subsection{Comparison}
Let us compare our estimates with \cite[Theorem 2, Remark 3]{fournier2015rate}, which states the following:
\begin{enumerate}
    \item If $\mu$ is supported on $(-1,1]^d$, then for all $N\in \N$ and $x>0$,
    \begin{align}
        \P(\t_p(\mu, \mu_N)>x)
        \le Ce^{-cN\varphi(x)}.
    \end{align}
    \item If $M_{q}(\mu)<\infty$ for some $q>2p$, then for all $N\in \N$, $\ep\in (0, q)$ and $x>0$,
    \begin{align}
        \P(\t_p(\mu, \mu_N)>x)
        \le C\bigg(e^{-cN\varphi(x)}\ind{\{x\le 1\}}
        +N(Nx)^{-(q-\ep)/p}\bigg).
    \end{align}
    \item If $\int_{\R^d} e^{a\abs{y}^{\beta}}\,d\mu(y)<\infty$ for some $\beta>p$ and $a>0$, then for all $N\in \N$ and $x>0$,
    \begin{align}
        \P(\t_p(\mu, \mu_N)>x)
        \le C\bigg(e^{-cN\varphi(x)}\ind{\{x\le 1\}}
        +e^{-cN x^{\beta/p}}\ind{\{x> 1\}}\bigg).
    \end{align}
    \item If $\int_{\R^d} e^{a\abs{y}^{\beta}}\,d\mu(y)<\infty$ for some $\beta\in(0,p)$ and $a>0$, then for all $N\in \N$ and $x>0$,
    \begin{align}
        \P(\t_p(\mu, \mu_N)>x)
        \le C\bigg(e^{-cN\varphi(x)}\ind{\{x\le 1\}}
        +e^{-cNx^2(\log(N+1))^{1-2p/\beta}}
        +e^{-c(Nx)^{\beta/p}}\bigg).
    \end{align}
    \item If $\int_{\R^d} e^{a\abs{y}^{\beta}}\,d\mu(y)<\infty$ for some $\beta\in(0,p)$ and $a>0$, then for all $N\in \N$, $\ep\in (0,\beta)$ and $x>0$,
    \begin{align}
        \P(\t_p(\mu, \mu_N)>x)
        \le C\bigg(e^{-cN\varphi(x)}\ind{\{x\le 1\}}
        +e^{-c(Nx)^{(\beta-\ep)/p}}\ind{\{x\le 1\}}
        +e^{-c(Nx)^{\beta/p}}\ind{\{x> 1\}}\bigg).
    \end{align}
\end{enumerate}
Example~\ref{ex:cpt} recovers the rates for compactly supported $\mu$. For all other cases, Theorem~\ref{thm:goal} applied to $\t_p$ improves on existing results in \cite{fournier2015rate}. Indeed, Example~\ref{ex:wsm}$\ref{ex:wsm hq}$ shows that one can take $\varepsilon=0$ in $(2)$. Furthermore, contrary to $(3)$, Example~\ref{ex:wsem}$\ref{ex:wsem hb}$ covers the case $\beta=p$. Finally, Example~\ref{ex:wsem}$\ref{ex:wsem lb}$ shows that the logarithmic term and the $\ep$-term can be removed in $(4)$ and $(5)$. Our paper additionally covers the case $M_q(\mu)<\infty$ with $q\le 2p$, which was not covered in \cite{fournier2015rate} at all. As a further indication of the sharpness of the estimates in Example~\ref{ex:wsm}, let us derive moment bounds for $\t_p$. Using the identity $\E[\t_p(\mu, \mu_N)]=\int_0^{\infty}\P(\t_p(\mu, \mu_N)>x)\,dx$, we obtain
\begin{align}
    \E[\t_p(\mu, \mu_N)]
    \le C
    \begin{dcases}
     N^{-1/2}+N^{-(q-p-\ep)/q} & \text{ if } p>d/2\,\\
    \log(N+1)N^{-1/2}+N^{-(q-p-\ep)/q} & \text{ if } p=d/2\,\\
    N^{-p/d}+N^{-(q-p-\ep)/q} & \text{ if } p\in (0,d/2)\,
    \end{dcases}
\end{align}
for arbitrarily small $\ep\in (0, q-p)$. Compared with \cite[Theorem 1]{fournier2015rate} these bounds only introduce an additional $\ep$-loss in the decay rate. Since the bounds in \cite[Theorem 1]{fournier2015rate} are known to be close to optimal, our rates must be close to optimal too.

\subsection{Estimates for \texorpdfstring{$\mathcal{E}_{p,a}$}{TEXT}}
Now let us consider $f(r)=e^{ar^p}-1$ for $a,p>0$. As mentioned in Section~\ref{sec:main results}, we take $g=f$ and $G(r)=\frac{1}{2}(e^{2c_p a r^p}-1)$ where $c_p=2^{p-1}$ if $p\ge 1$ and $c_p=1$ if $p\in (0,1)$. Let us assume that $\int_{\R^d} e^{b\abs{y}^p}\,d\mu(y)<\infty$ for some $b>2^{p+1}c_p a$. This allows us to take $\gamma=b/(2c_p a)$ and $S(r)=e^{b r^p}$. In order to have $\mathcal{K}_g<\infty$ and $\mathcal{K}_G<\infty$, we select $\eta\in (0, 1-2^{p+1}c_pa/b)\cap (0,\frac{1}{2} \wedge \frac{p}{d}]$. Lastly we set $c_0=1$. As before we now determine the maximal value $\eta$ satisfying these constraints. For $p\ge 1$ we obtain the following:
\begin{enumerate}
    \item If $p\ge d/2$ and $b>2^{2p+1}a$, then $\gamma>2$ and $\eta=1/2$.
    \item If $p\ge d/2$ and $b\in (4^{p} a, 2^{2p+1}a]$, then $\gamma>2$ and $\eta=1-4^{p}a/b-\ep$ for arbitrarily small $\ep\in (0, 1-4^{p}a/b)$.
    \item If $p\in (0,d/2)$ and $b> 4^{p}da/(d-p)$, then $\gamma>2$ and $\eta=p/d$.
    \item If $p\in (0,d/2)$ and $b\in (4^{p}a, 4^{p}da/(d-p)]$, then $\gamma>2$ and $\eta=1-4^{p}a/b-\ep$ for arbitrarily small $\ep\in (0,1-4^{p}a/b)$.
\end{enumerate}
The computations for $p\in (0,1)$ are essentially the same as for $p\ge 1$, so we skip the details. Using Proposition~\ref{prop:dvp}$\ref{prop:mg}$--$\ref{prop:ml}$, we then compute the estimates in Theorem~\ref{thm:goal}. We summarize the results below.

\begin{ex}[$\mathcal{E}_{p, a}$ when $p\ge 1$]\label{ex:expdhp} Let $p\ge 1$ and $E:=\int_{\R^d} e^{b\abs{y}^p}\,d\mu(y)<\infty$ for some $b>4^{p}a$.
\begin{anumerate}
    \item\label{ex:expdhp hb} Suppose that $b>2^{2p+1}a$ if $p\ge d/2$ and $b>4^{p}da/(d-p)$ if $p\in (0,d/2)$. Then for all $N\in \N$ and $x>0$,
    \begin{align}
        \P(\mathcal{E}_{p, a}(\mu, \mu_N)>x)
        \le C\bigg(e^{-cN\varphi(x)}\ind{\{x\le 1\}}
        +N(Nx)^{-b/(2^p a)}\bigg).
    \end{align}
    \item\label{ex:expdhp lb} Suppose that $b\in (4^{p}a, 2^{2p+1}a]$ if $p\ge d/2$ and $b\in (4^{p}a, 4^{p}da/(d-p)]$ if $p\in (0,d/2)$. Then for all $N\in \N$, $\ep\in (0,1-4^{p}a/b)$ and $x>0$,
    \begin{align}
        \P(\mathcal{E}_{p, a}(\mu, \mu_N)>x)
        \le C\bigg(e^{-cN\varphi_{1-4^{p}a/b-\ep}(x)}\ind{\{x\le 1\}}
        +N(Nx)^{-b/(2^p a)}\bigg).
    \end{align}
\end{anumerate}
Here, the constants $c>0$ and $C>0$ depend only on $d,p,a,b, E$ and if applicable, also on $\ep$.
\end{ex}

\begin{ex}[$\mathcal{E}_{p,a}$ when $p\in (0,1)$]\label{ex:expdlp} Let $p\in (0,1)$ and $E:=\int_{\R^d}e^{b\abs{y}^p}\,d\mu(y)<\infty$ for some $b>2^{p+1}a$.
\begin{anumerate}
    \item Suppose that
    \begin{align}
        b
        >
        \begin{dcases}
        2^{p+2}a & \text{ if } p\ge d/2\,\\
        4a & \text{ if } p\in (0,d/2),\, 2^pd/(d-p)\le 2\,\\
        2^{p+1}da/(d-p) & \text{ if } p\in (0,d/2),\, 2^{p}d/(d-p)>2.
    \end{dcases}
    \end{align}
    Then for all $N\in \N$ and $x>0$,
    \begin{align}
        \P(\mathcal{E}_{p,a}(\mu, \mu_N)>x)
        \le C\bigg(e^{-cN\varphi(x)}\ind{\{x\le 1\}}
        +N(Nx)^{-b/(2a)}\bigg).
    \end{align}
    \item Suppose that
    \begin{align}
        \begin{dcases}
        4a< b\le 2^{p+2}a & \text{ if } p\ge d/2\,\\
        2^{p+1}da/(d-p)< b\le 4a & \text{ if } p\in (0,d/2),\, 2^{p}d/(d-p)\le 2\,\\
        4a< b\le 2^{p+1}da/(d-p) & \text{ if } p\in (0,d/2),\, 2^p d/(d-p)>2.
        \end{dcases}
    \end{align}
    Then for all $N\in \N$, $\ep\in (0, 1-2^{p+1}a/b)$ and $x>0$,
    \begin{align}
        \P(\mathcal{E}_{p,a}(\mu, \mu_N)>x)
        \le C\Big(a(N,x)\ind{\{x\le 1\}}
        +N(Nx)^{-b/(2a)}\Big),
    \end{align}
    where
    \begin{align}
        a(N, x)
        :=\begin{dcases}
        e^{-cN\varphi_{1-2^{p+1}a/b-\ep}(x)} & \text{ if } p\ge d/2\,\\
        e^{-cN\varphi(x)}
        +e^{-cN^{2(1-2a/b-\ep)}x^2}& \text{ if } p\in (0,d/2),\, 2^{p}d/(d-p)\le 2\,\\
        e^{-cN\varphi_{1-2^{p+1}a/b-\ep}(x)} &  \text{ if } p\in (0,d/2),\, 2^{p}d/(d-p)> 2.
        \end{dcases}
    \end{align}
    \item Suppose that
    \begin{align}
        2^{p+1}a
        <b
        \le
        \begin{dcases}
        4a & \text{ if } p\ge d/2\,\\
        2^{p+1}da/(d-p) & \text{ if } p\in (0,d/2),\, 2^{p}d/(d-p)\le 2\,\\
        4a & \text{ if } p\in (0,d/2),\, 2^p d/(d-p)>2.
        \end{dcases}
    \end{align}
    Then for all $N\in \N$, $\ep\in (0, 1-2^{p+1}a/b)$ and $x>0$,
    \begin{align}
        &\P(\mathcal{E}_{p,a}(\mu, \mu_N)>x)\\
        &\le C\bigg(e^{-cN\varphi_{1-2^{p+1}a/b-\ep}(x)}\ind{\{x\le 1\}}
        +e^{-cN^{2(1-2a/b-\ep)}x^2}\ind{\{x\le 1\}}
        +N(Nx)^{-b/(2a)}\bigg).
    \end{align}
\end{anumerate}
Here, the constants $c>0$ and $C>0$ depend only on $d,p,a,b,E$ and if applicable, also on $\ep$.
\end{ex}

As for $\t_p$ we can compute moment bounds of $\mathcal{E}_{p, a}$ from Examples~\ref{ex:expdhp} and~\ref{ex:expdlp}. Recalling $\E[\mathcal{E}_{p, a}(\mu, \mu_N)]=\int_0^{\infty}\P(\mathcal{E}_{p, a}(\mu, \mu_N)>x)\,dx$ we obtain the following bounds.
\begin{ex}[Moment bounds of $\mathcal{E}_{p, a}$ when $p\ge 1$]\label{ex:mom expdhp} Let $p\ge 1$ and $E:=\int_{\R^d}e^{b\abs{y}^p}\,d\mu(y)<\infty$ for some $b>4^{p}a$.
\begin{anumerate}
    \item\label{ex:mom expdhp hb} Under the same assumptions as in Example~\ref{ex:expdhp}\ref{ex:expdhp hb} we have
    \begin{align}
        \E[\mathcal{E}_{p,a}(\mu, \mu_N)]
        \le C
        \begin{dcases}
            N^{-1/2} & \text{ if } p>d/2\,\\
            \log(N+1)N^{-1/2} & \text{ if } p=d/2\,\\
            N^{-p/d} & \text{ if } p\in (0,d/2).
        \end{dcases}
    \end{align}
    \item Under the same assumptions as in Example~\ref{ex:expdhp}\ref{ex:expdhp lb} and with the same $\ep$ we have
    \begin{align}
        \E[\mathcal{E}_{p,a}(\mu, \mu_N)]
        \le C N^{-(1-4^pa/b-\ep)}.
    \end{align}
\end{anumerate}
Here, the dependence of $C$ on parameters is the same as in Example~\ref{ex:expdhp}.
\end{ex}
\begin{ex}[Moment bounds for $\mathcal{E}_{p,a}$ when $p\in (0,1)$]\label{ex:mom expdlp} Let $p\in (0,1)$ and $E:=\int_{\R^d}e^{b\abs{y}^p}\,d\mu(y)<\infty$ for some $b>2^{p+1}a$.
\begin{anumerate}
    \item\label{ex:mom expdlp hb} Suppose that $b>2^{p+2}a$ if $p\ge d/2$ and $b>2^{p+1}da/(d-p)$ if $p\in (0,d/2)$. Then for all $N\in \N$,
    \begin{align}
        \E[\mathcal{E}_{p,a}(\mu, \mu_N)]
        \le C
        \begin{dcases}
        N^{-1/2} & \text{ if } p>d/2\,\\
        \log(N+1)N^{-1/2} & \text{ if } p=d/2\,\\
        N^{-p/d} & \text{ if } p\in (0,d/2).
    \end{dcases}
    \end{align}
    \item Suppose that $b\in (2^{p+1}a, 2^{p+2}a]$ if $p\ge d/2$ and $b\in (2^{p+1}a, 2^{p+1}da/(d-p)]$ if $p\in (0,d/2)$. Then for all $N\in \N$ and $\ep\in (0, 1-2^{p+1}a/b)$,
    \begin{align}
        \E[\mathcal{E}_{p,a}(\mu, \mu_N)]
        \le C N^{-(1-2^{p+1}a/b-\ep)}.
    \end{align}
\end{anumerate}
Here, the dependence of $C$ on the parameters is the same as in Example \ref{ex:expdlp}.
\end{ex}
The moment bounds in Examples~\ref{ex:mom expdhp}$\ref{ex:mom expdhp hb}$ and~\ref{ex:mom expdlp}$\ref{ex:mom expdlp hb}$ are sharp. This follows from the fact that $c_1\t_p\le \mathcal{E}_{p,a}\le c_2 \t_p$ for compactly supported measures together with \cite[Section 1]{fournier2015rate}, where the existence of compactly supported measures $\mu, \nu\in \sp(\R^d)$ such that $\E[\t_p(\mu, \mu_N)]\ge CN^{-1/2}$ and $\E[\t_p(\nu, \nu_N)]\ge CN^{-p/d}$ is stated. Furthermore, it follows from \cite{ajtai1984optimal} that the uniform distribution $\rho$ on $[-1,1]^d$ admits the lower bound $\E[\t_p(\rho, \rho_N)]\ge C (\log N/N)^{1/2}$ when $p=d/2=1$.

\section{Numerical Tests}\label{sec:numericaltest}

\begin{figure}[ht]
    \centering
    \includegraphics[width = 12.6cm, height = 8.4cm]{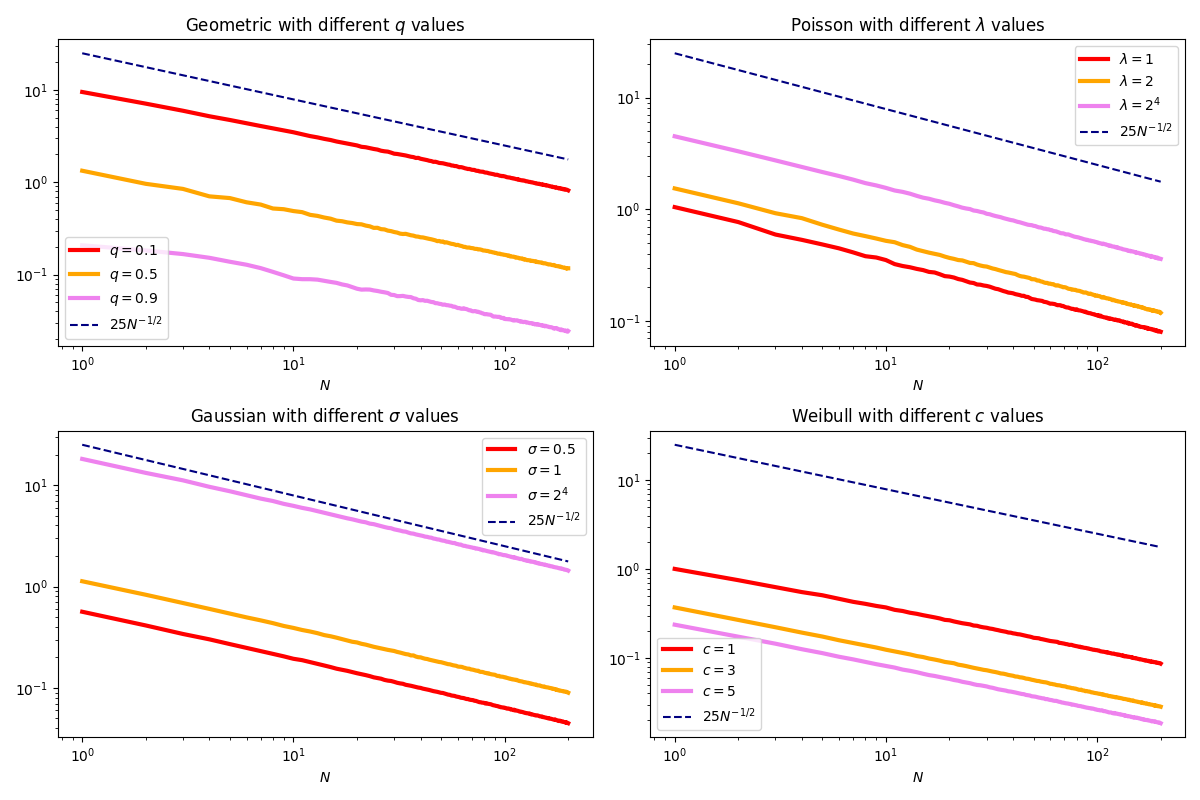}
    \caption{Log-log plots of functions $N\mapsto \E[\t_p(\mu, \mu_N)]$ for Geometric, Poisson, Gaussian and Weibull distributions with $p=1$. The dashed line is a theoretical upper bound $C/\sqrt{N}$ with $C=25$.} 
    \label{fig:wass_diff_distri}
\end{figure}

\begin{figure}[ht]
    \centering
    \includegraphics[width = 12.6cm, height = 8.4cm]{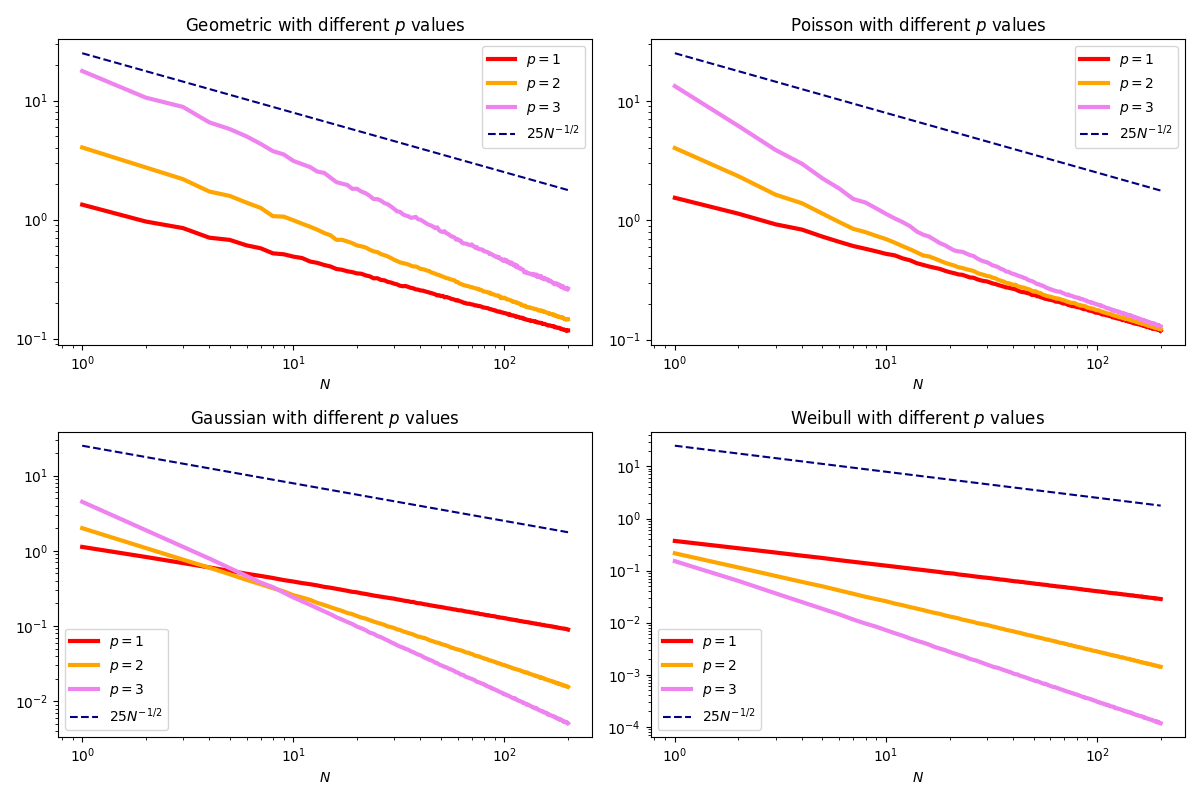}
    \caption{Log-log plots of functions $N\mapsto \E[\t_p(\mu, \mu_N)]$ for Geometric, Poission, Gaussian and Weibull distributions with $p=1,2,3$. The dashed line is a theoretical upper bound $C/\sqrt{N}$ with $C=25$.}
    \label{fig:wass_diff_p}
\end{figure}

In this section, we test our theoretical moment bounds numerically. We focus here on the one dimensional case, i.e. $d=1$, where it is known that the monotone coupling is optimal for radial and convex nonnegative cost functions; see \cite[Remark 2.19]{villani2021topics} for details. In particular, the monotone coupling is optimal both for $\t_p$ and $\mathcal{E}_{p,a}$ when $p\ge 1$. Using this fact we compute $\E[\t_p(\mu, \mu_N)]$ and $\E[\mathcal{E}_{p,a}(\mu, \mu_N)]$ as a function of $N$ by Monte Carlo simulation, averaging over $10^4$ trials. We take sample sizes $N$ in the range $[1,200]$.\footnote{See \url{https://github.com/Jonghwap/concentration2023/}.}

In Figure \ref{fig:wass_diff_distri} we compare $\E[\t_1(\mu, \mu_N)]$ with the theoretical upper bound $C/\sqrt{N}$. For ease of presentation, results are presented on a log-log scale. In the top left graph we take $\mu\sim \text{Geom}(q)$, i.e. $\mu(\{k\})=(1-q)^{k-1}q$, $k\ge 1$ and plot $\E[\t_1(\mu, \mu_N)]$ for $q=0,1, 0.5, 0.9$. The top right plot in Figure \ref{fig:wass_diff_distri} showcases results for $\mu\sim \text{Poiss}(\lambda)$ and $\lambda=1, 2, 2^4$. In the bottom of Figure \ref{fig:wass_diff_distri}, we present results for $\mu\sim N(0, \sigma^2)$, $\sigma=1/2, 1, 2^4$ on the left, and for $\mu$ equal to a Weibull distribution with parameter $c=1, 3, 5$ on the right, i.e. $\mu(dx)= c x^{c-1}e^{-x^{c}}\ind{\{x>0\}}dx$. The numerical results are consistent with the theoretical upper bound $C/\sqrt{N}$.
In fact the functions $N\mapsto \E[\t_1(\mu, \mu_N)]$ are almost parallel to the dashed line for large $N$. This is numerical evidence that the convergence rate $1/\sqrt{N}$ is almost sharp for these distributions.

In Figure \ref{fig:wass_diff_p} we now plot $N\mapsto \E[\t_p(\mu, \mu_N)]$ for different values of $p$. We consider the same distributions as before. In particular, the top left plot shows results for $\mu\sim \text{Geom}(1/2)$, while $\mu\sim \text{Poiss}(2)$ in the top right plot. The bottom left plot shows results for $\mu\sim N(0,1)$, while $\mu$ is equal to a Weibull distribution with the parameter $c=3$ in the bottom right plot. We take $p=1, 2, 3$. As in Figure \ref{fig:wass_diff_distri}, numerical results in Figure \ref{fig:wass_diff_p} are consistent with the upper bound $C/\sqrt{N}$. Also, the convergence rate $1/\sqrt{N}$ seems almost sharp for these distributions when $p=1$. However, when $p>1$, numerical results no longer suggest that the convergence rate $1/\sqrt{N}$ is sharp except for a Poisson distribution (top right). %
For all other distributions, the slope of a function $N\mapsto \E[\t_p(\mu, \mu_N)]$ is estimated to be strictly smaller than $-1/2$, implying that the inferred convergence rate is $1/N^{\alpha}$ for $\alpha>1/2$. In fact, $\alpha$ is estimated to increase as $p$ increases.

\begin{figure}[ht]
    \centering
    \includegraphics[width = 12.6cm, height = 8.4cm]{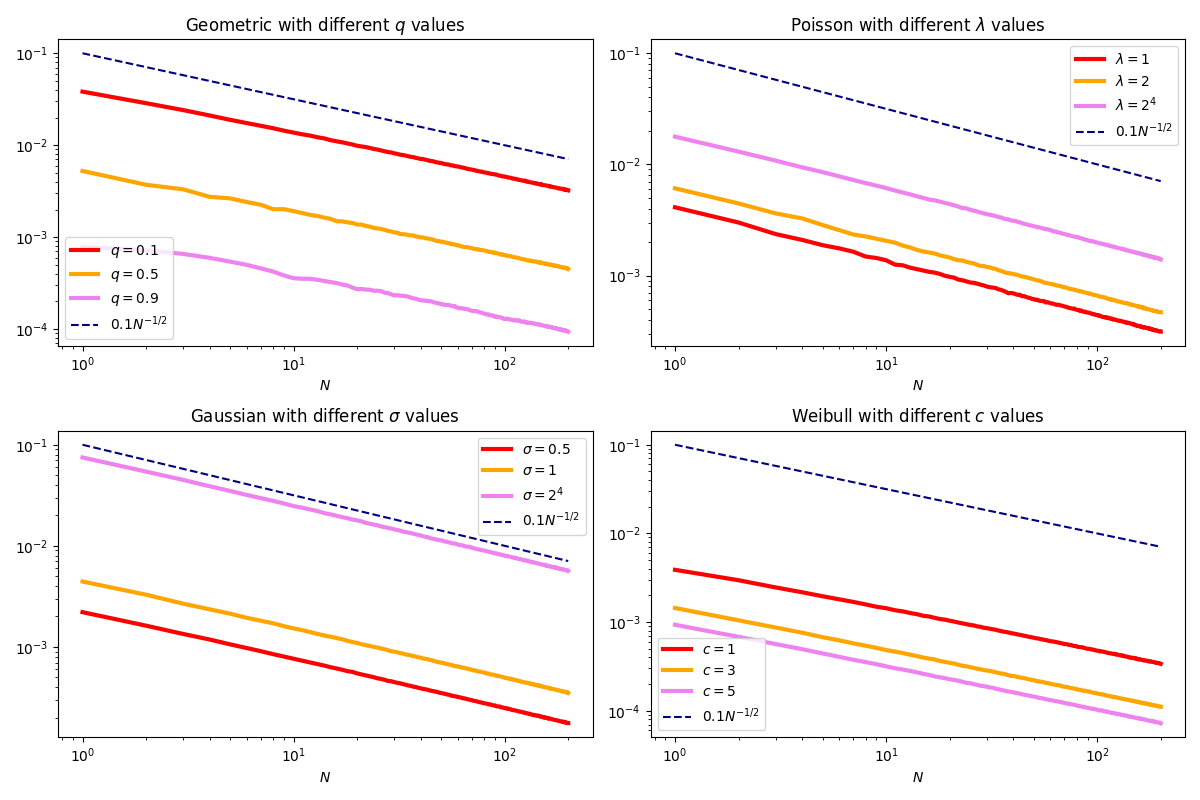}
    \caption{Log-log plots of functions $N\mapsto \E[\mathcal{E}_{p, a}(\mu, \mu_N)]$ for Geometric, Poisson, Gaussian and Weibull distributions with $p=1$ and $a=1/2^8$. The dashed line is a theoretical upper bound $C/\sqrt{N}$ with $C=0.1$.}
    \label{fig:diff_distri}
\end{figure}

\begin{figure}[ht]
    \centering
    \includegraphics[width = 12.6cm, height = 6.048cm]{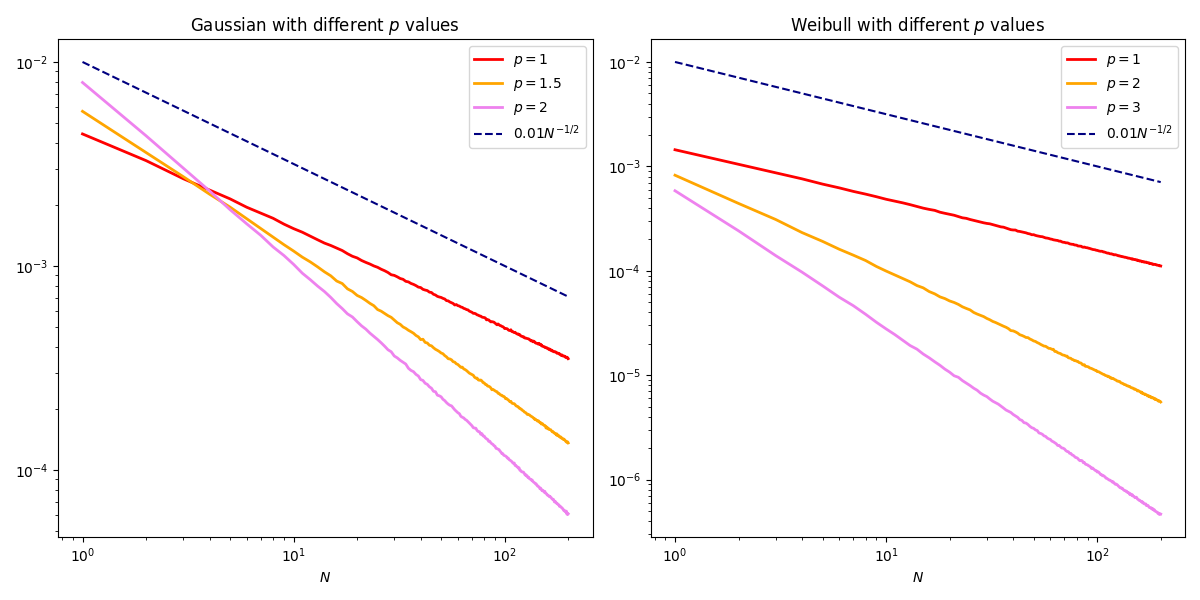}
    \caption{Log-log plots of functions $N\mapsto \E[\mathcal{E}_{p,a}(\mu, \mu_N)]$ for Gaussian and Weibull distributions with $p=1,2,3$ and $a=1/2^8$. The dashed line is a theoretical upper bound $C/\sqrt{N}$ with $C=0.01$.}
    \label{fig:diff_p}
\end{figure}

We present numerical results for $N\mapsto \E[\mathcal{E}_{p,a}(\mu, \mu_N)]$ in Figures~\ref{fig:diff_distri} and \ref{fig:diff_p}. In Figure~\ref{fig:diff_distri}, we compare $N \mapsto \E[\mathcal{E}_{1,a}(\mu, \mu_N)]$ with its theoretical upper bound $C/\sqrt{N}$. We test the same set of distributions as in Figure~\ref{fig:wass_diff_distri}. To satisfy the moment conditions in Example~\ref{ex:mom expdhp}, we take $a$ sufficiently small, say $a=1/2^8=1/256$. Numerical results in Figure~\ref{fig:diff_distri} are consistent with the theoretical upper bound $C/\sqrt{N}$ for $C\ge 0.1$. As before, examining the slope of the functions we conclude that $C/\sqrt{N}$ is a sharp upper bound of $N\mapsto \E[\mathcal{E}_{p,a}(\mu, \mu_N)]$  for these distributions.

Lastly we carry out an analysis similar to Figure \ref{fig:wass_diff_p} for $N\mapsto \E[\mathcal{E}_{p,a}(\mu, \mu_N)]$. For this we take $\mu\sim N(0,1)$ in the left plot and consider a Weibull distribution with the parameter $c=3$ in the right plot. Again we take $p=1,2,3$ with $a=1/2^8$. It is easy to check that $a=1/2^8$ satisfies the moment conditions. On the other hand, neither the geometric nor the Poisson distributions have exponential moments for $p>1$. Figure \ref{fig:diff_p} shows that the computed values are consistent with the upper bound $C/\sqrt{N}$ for $C\ge 0.01$. As observed in Figure \ref{fig:wass_diff_p}, the function $N\mapsto \E[\mathcal{E}_{p,a}(\mu, \mu_N)]$ seems to converge faster than $1/\sqrt{N}$ for $p>1$.

\section{Proof of Theorem \ref{thm:goal}}\label{sec:proof}

This section is devoted to the proof of Theorem~\ref{thm:goal}. A key step in the proof is to partition the support of $\mu$ into a countable number of annuli. The notation required to formalize this construction is introduced in Subsection~\ref{sec_partitions}. In Subsection~\ref{sec_sketch} we provide a high-level sketch of the proof, while the formal argument is given in Subsection~\ref{sec_formal_proof}.

\subsection{Partitions} \label{sec_partitions}

Let $B_r(x)$ and $\overline{B_r(x)}$ denote the open respectively closed ball of radius $r>0$ centered at $x\in \R^d$. Our proof will rely on a partition $(A^k)_{k\in \N}$ of $\R^d$ which is defined via $A^0:=\overline{B_1(0)}$ and $A^k:=\overline{B_{2^k}(0)}\setminus \overline{B_{2^{k-1}}(0)}$ for $k\ge 1$. For $\mu\in \sp(\R^d)$ and its empirical measure $\mu_N$, let us define their conditional measures on $A^k$ via
\begin{align}
    \mu^{k}
    :=
    \begin{dcases} \mu(\,\cdot\,\st A^k) & \text{ if } \mu(A^k)>0\,\\
    \delta_0 & \text{ if } \mu(A^k)=0,
    \end{dcases}\quad
    \mu_N^k
    :=
    \begin{dcases} \mu_N(\,\cdot\,\st A^k) & \text{ if } \mu_N(A^k)>0\,\\
    \delta_0 & \text{ if } \mu_N(A^k)=0.
    \end{dcases}
\end{align}
We note that if $\mu(A^k)>0$, $\mu^k$ is compactly supported on $A^k$. Similarly, on the event $\{\mu_N(A^k)>0\}$, $\mu^k_N$ is supported on $A^k$. Next we define the dilation $T^k\colon\R^d\to \R^d$ via $T^k(x):=x/2^k$. We denote the pushforward measure of $\mu^k$ through $T^k$ by  $T^k_{\#}\mu^k$ and remark that it is supported on $A^0$.

For i.i.d.~samples $Y_1, Y_2,\dots$ of $\mu^k$ we denote the empirical measure of $\mu^k$ with a sample size $N\in \N$ by $(\mu^k)_N := (1/N)\sum_{j=1}^{N}\delta_{Y_j}$ if $N\ge 1$ and $(\mu^k)_N:=\delta_0$ if $N=0$ as usual. Note that $(\mu^k)_N$ differs from the random measure $\mu^k_N$ defined above; the former is based on $N$ point masses in $A^k$, whereas the latter is based on a random number $N \mu_N(A^k)$ of point masses in $A^k$. 

Let $\g_N$ be the $\sigma$-algebra generated by $\{\mu_N(A^k)\}_{k\in \N}$. Since $\mu_N(A^k)$ is a discrete random variable, $\g_N$ can be characterized explicitly as follows. Defining
\begin{align}
    \mathbf{L}_N:=\left\{(L_{k})_{k\in \N} \st L_{k}\in \N,\,\, \sum_{k=0}^{\infty}L_{k}=N\right\} \text{ for all } N\in \N\,
\end{align}
the events $\{N\mu_N(A^k)=L_{k} \text{ for all }k\in \N\}$ with $(L_k)_{k \in \N} \in \mathbf{L}_N$ constitute the atoms of $\g_N$ (with a slight abuse of terminology as some of these ``atoms'' could be empty.) Hence for every $\P$-integrable random variable $Z$ we have
\begin{align}
    \E[Z \st \g_N]
    =\sum_{(L_{k})_{k\in \N}\in \mathbf{L}_N}
    \E^{(L_{k})}[Z]\ind{\{N\mu_N(A^k)=L_{k} \text{ for all } k\in \N\}},
\end{align}
where $\E^{(L_k)}$ is the expectation under the probability $\P^{(L_k)}$ defined via
\begin{align}
    \P^{(L_{k})}
    =\begin{dcases}
    \P(\,\cdot \, \st N\mu_N(A^k)=L_{k} \text{ for all } k\in \N) & \text{ if } \P(N\mu_N(A^k)=L_{k} \text{ for all } k\in \N)>0\,\\
    \P & \text{ otherwise.}
    \end{dcases}
\end{align}
Throughout the proof, all generic constants, which depend only on the external parameters $d$, $p$, $\gamma$, $\eta$, $\delta_0$, $\mathcal{K}_g$, $\mathcal{K}_G$, are allowed to change from line to line.

\subsection{Sketch of the proof of Theorem \ref{thm:goal}} \label{sec_sketch}

Before we proceed to give a rigorous proof of Theorem~\ref{thm:goal}, let us sketch some of its main arguments. For simplicity, we consider the case where $\d_f=\t_p$ and assume that $p>d/2$ and $M_q(\mu)<\infty$ for some $q>2p$. By choosing a suitable coupling between $\mu$ and $\mu_N$, we estimate
\begin{equation} \label{eq:sketch}
\begin{aligned}
    \t_p(\mu, \mu_N)
    &\le \sum_{k=0}^{\infty}(\mu(A^k)\wedge \mu_N(A^k))\t_p(\mu^k, \mu^k_N)
    +\sum_{k=0}^{\infty}(\mu(A^k)\wedge\mu_N(A^k))(M_p(\mu^k)-M_p(\mu^k_N))\\
    &+2\sum_{k=0}^{\infty}(\mu(A^k)-\mu_N(A^k))_{+}M_p(\mu^k)
    +M_p(\mu_N)-M_p(\mu).
\end{aligned}
\end{equation}
A similar estimate is used in \cite[Lemma 5]{fournier2015rate}; we refer to the proof of Theorem \ref{thm:goal} for details. The proof proceeds by estimating each of the above summands in \eqref{eq:sketch} separately. 

The first two series can be estimated using rates for compactly supported measures. Indeed, $\mu^k$ is compactly supported on either $A^k$ or $A^0$. Moreover we show in Lemma~\ref{lem:cond} that conditionally on $\g_N$, $\mu^k_N$ has the same distribution as $(\mu^k)_{N\mu_N(A^k)}$, i.e.\ an empirical measure of $\mu^k$ with a sample size $N\mu_N(A^k)$ which is random but constant on each atom of $\g_N$. A standard concentration inequality for the transport cost between $\mu$ and $\mu_N$ for compactly supported measures (see Lemma~\ref{lem:cmpbd}) combined with a change of variables then shows that
\begin{align}
    \P(2^{-kp}\t_p(\mu^k, \mu^k_N)>y\st \g_N)
    =\P(\t_p(T^{k}_{\#}\mu^k, T^{k}_{\#}(\mu^k)_{N\mu_N(A^k)})>y\st \g_N)
    \le Ce^{-cN\mu_N(A^k)y^2}.
\end{align}
Setting $x=(\mu(A^k)\wedge \mu_N(A^k))^{1/2}y$ we obtain
\begin{align}
    \P(2^{-kp}(\mu(A^k)\wedge \mu_N(A^k))^{1/2}\t_p(\mu^k, \mu^k_N)>x)\le Ce^{-cN x^2}.\label{eq:skt1}
\end{align}
With more effort, one proves that the estimate \eqref{eq:skt1} remains true after taking a supremum over $k\in \N$ inside the probability in \eqref{eq:skt1}. Hence (see Lemma~\ref{lem:cpt union bd} for more details),
\begin{align}
    \P\left(\sum_{k=0}^{\infty}(\mu(A^k)\wedge \mu_N(A^k))\t_p(\mu^k, \mu^k_N)>x\sum_{k=0}^{\infty}2^{kp}(\mu(A^k))^{1/2}\right)
    \le Ce^{-cNx^2}.\label{eq:skt2}
\end{align}
Here, the constant $\sum_{k=0}^{\infty}2^{kp} (\mu(A^k))^{1/2}$ can be controlled by $M_q(\mu)$ (more precisely it is bounded above by $C(1\vee M_q(\mu))^{1/2}$). The second series in \eqref{eq:sketch} can be controlled in a similar way; we refer to Lemma~\ref{lem:cpt union bd2} for details.

In order to bound the third series in \eqref{eq:sketch} we use techniques from empirical process theory. For this we first observe that
\begin{align}
    \P\left(\sum_{k=0}^{\infty}(\mu(A^k)-\mu_N(A^k))_{+}M_p(\mu^k)>x\sum_{k=0}^{\infty}\sqrt{\mu(A^k)}M_p(\mu^k)\right)
    \le \P\left(\sup_{k\in \N}\frac{(\mu(A^k)-\mu_N(A^k))_{+}}{\sqrt{\mu(A^k)}}>x\right).
\end{align}
Empirical process theory allows us to bound the right hand side by $C e^{-cN x^2}$, see Lemma~\ref{lem:relbd} for details. 

Combining the above estimates for the three series on the right hand side of \eqref{eq:sketch}, we finally find
\begin{align}
    \P(\t_p(\mu, \mu_N)>(1\vee M_{q}(\mu))^{1/2}x)\le Ce^{-cN x^2}+\P(M_p(\mu_N)-M_p(\mu)>x)\,
\end{align}
which was the claim.

\subsection{Proof of Theorem \ref{thm:goal}} \label{sec_formal_proof}

The proof is organized as follows: Lemmas~\ref{lem:compute bd} and~\ref{lem:phi} state some useful technical results that will be used repeatedly. We then record the Fournier--Guillin concentration inequality between $\mu$ and $\mu_N$ for compactly supported measures in Lemma~\ref{lem:cmpbd}. Lemmas~\ref{lem:cond},~\ref{lem:cpt union bd} and~\ref{lem:cpt union bd2} are extensions of Lemma~\ref{lem:cmpbd}, and adapt it to our setting. Lemma~\ref{lem:relbd} proves a concentration inequality for self-normalized empirical processes. Lastly we combine these results to complete the proof of Theorem~\ref{thm:goal}. Throughout this section, the notation established in Subsections~\ref{S_notation} and~\ref{sec_partitions} is used freely.

\begin{lem}\label{lem:compute bd} Let $Y_{k, N}$ be a random variable indexed by $k, N\in \N$. Suppose there exists a Borel measurable function $h\colon (0,\infty)\to (0,\infty)$ and constants $A, \beta>0$, $a>1$ such that
\begin{align}
    \P(Y_{k, N}>x)
    \le Ae^{-(1+k\log(a))N^{\beta}h(x)}
    =A a^{-kN^{\beta}h(x)} e^{-N^{\beta}h(x)} \text{ for all } k, N\in \N\text{ and } x>0.
\end{align}
Then
\begin{align}
    \P\left(\sup_{k\in \N}Y_{k, N}>x\right)
    \le \left(e\vee \frac{A}{1-a^{-1}}\right)e^{- N^{\beta}h(x)} \text{ for all } N\in \N\text{ and } x>0.
\end{align}
\end{lem}
\begin{proof}
By a union bound, we have
\begin{align}
    \P\left(\sup_{k\in \N}Y_{k, N}>x\right)
    \le Ae^{-N^{\beta}h(x)}\sum_{k=0}^{\infty}a^{-kN^{\beta}h(x)}
    =\frac{A}{1-a^{- N^{\beta}h(x)}}e^{- N^{\beta}h(x)}.
\end{align}
If $N^{\beta}h(x)>1$, $\frac{A}{1-a^{- N^{\beta}h(x)}}<\frac{A}{1-a^{-1}}$ follows, which proves the desired estimate. If $N^{\beta}h(x)\le 1$ we can bound
\begin{align}
    \P\left(\sup_{k\in \N}Y_{k, N}>x\right)
    \le 1
    =ee^{-1}
    \le e\cdot e^{-N^{\beta}h(x)}.
\end{align}
This finishes the proof.
\end{proof}

Next we prove some properties of the modified rate function $\varphi_\eta$ defined in \eqref{eq_phi_eta}.

\begin{lem}\label{lem:phi} Let $0<\eta\le \frac{1}{2} \wedge \frac{p}{d}$ and $x>0$. Then the following hold:
\begin{anumerate}
    \item\label{lem:phi a} If $a\ge 1$, $\varphi_{\eta}(ax)\ge a^2\varphi_{\eta}(x)$.
    \item\label{lem:phi b} If $x^{1/\eta}\le a\le 1$, $a\varphi(x/a^{\eta})\ge \varphi_{\eta}(x)$. 
    \item\label{lem:phi c} If $a\ge x$, $\varphi_{\eta}(x/a)\ge A\varphi_{\eta}(x)$ for some constant $A>0$ depending only on $a$ and $\eta$.
\end{anumerate}
\end{lem}

\begin{proof}
$\ref{lem:phi a}$: Let $a\ge 1$. If $p\neq d/2$, then $\varphi_{\eta}(ax)=a^{1/\eta}\varphi_{\eta}(x)\ge a^2\varphi_{\eta}(x)$. If $p= d/2$, then $$\varphi_{\eta}(ax)=(ax)^{1/\eta}/(\log(1+1/(ax)))^2\ge (ax)^{1/\eta}/(\log(1+1/x))^2= a^{1/\eta} \varphi_{\eta}(x)\ge a^{2}\varphi_{\eta}(x).$$ This proves $\ref{lem:phi a}$.

$\ref{lem:phi b}$: Let $x^{1/\eta}\le a\le 1$. Using the fact that $1-2\eta\ge 0$ if $p\ge d/2$ and $1-\eta d/p\ge 0$ if $p\in (0,d/2)$, we compute
\begin{align}
    a\varphi(x/a^{\eta})
    =\begin{dcases}
    a^{1-2\eta}x^2\ge (x^{1/\eta})^{1-2\eta}x^2=\varphi_{\eta}(x) & \text{ if } p>d/2\,\\
    \frac{a^{1-2\eta}x^2}{(\log(2+a^{\eta}/x))^2}\ge \frac{(x^{1/\eta})^{1-2\eta}x^2}{(\log(2+1/x))^2}=\varphi_{\eta}(x) & \text{ if } p=d/2\,\\
    a^{1-\eta d/p}x^{d/p}\ge (x^{1/\eta})^{1-\eta d/p}x^{d/p}=\varphi_{\eta}(x) & \text{ if } p\in(0,d/2).
    \end{dcases}
\end{align}
This proves $\ref{lem:phi b}$.

$\ref{lem:phi c}$: Let $a\ge x$. Note that
\begin{align}
    \log(2+a/x)
    \le \log(2(a\vee 1)+(a\vee 1)/x)
    &=\log(a\vee 1)+\log(2+1/x)\\
    &\le \log(2+1/x)(\log(a\vee 1)/\log(2+1/a)+1).
\end{align}
Thus, if $p=d/2$, $\varphi_{\eta}(x/a)\ge a^{-1/\eta}(\log(2+1/a)/(\log(a\vee 1)+\log(2+1/a)))^2\varphi_{\eta}(x)$. If $p\neq d/2$, set $A=a^{-1/\eta}$.
\end{proof}

Next we state the Fournier--Guillin rate for compactly supported measures. Indeed, combining \cite[Proposition 10]{fournier2015rate} with \cite[Lemma 5]{fournier2015rate}, the following lemma easily follows from a dilation argument.

\begin{lem}[Fournier--Guillin, local estimate]\label{lem:cmpbd}
Let $p\in (0,\infty)$ and $d\ge 1$. Let us fix $R>0$. Then there exist constants $c>0$ depending only on $d,p,R$ and $C>0$ depending only on $d,p$ such that for any $\mu\in \sp(\R^d)$ supported on $\overline{B_R(0)}$ one has
\begin{align}
    \P(\t_p(\mu, \mu_N) > x)
    \le Ce^{-cN \varphi(x)}\ind{\{x\le (2R)^p\}}\,
\end{align}
for all $x>0$ and $N\in \N$.
\end{lem}

In the next lemma we prove that conditionally on $\g_N$, $\mu^k_N$ has the same distribution as $(\mu^k)_{N\mu_N(A^k)}$, which denotes the empirical measure of $\mu^k$ with a sample size $N\mu_N(A^k)$. Recall that $A^k$ are the annuli defined in Subsection~\ref{sec_partitions}.

\begin{lem}\label{lem:cond} Let $N\in \N$ and $(L_{k})_{k\in \N}\in \mathbf{L}_N$. Conditionally on $\g_N$, $(\mu^k_N)_{k\in \N}$ are independent. Moreover, the distribution of $\mu^k_N$ under $\P^{(L_{k})}$ is equal to the distribution of $(\mu^k)_{L_{k}}$.
\end{lem}

\begin{proof}
Let $(L_{k})_{k\in \N}\in \mathbf{L}_N$. To prove the claim it suffices to show that for all measurable sets $B_k \subseteq \sp(\R^d)$,
\begin{align}
    \P\left(\mu^k_N\in B_k \text{ for all }k\in \N \st N\mu_N(A^k)=L_{k} \text{ for all } k\in \N\right)
    =\prod_{k=0}^{\infty}\P\left( (\mu^k)_{L_{k}}\in B_k\right)\,
\end{align}
whenever $\P(N\mu_N(A^k)=L_{k} \text{ for all } k\in \N)>0$. Throughout the proof, let us fix $(L_{k})_{k\in \N}\in \mathbf{L}_N$ such that $\P(N\mu_N(A^k)=L_{k} \text{ for all } k\in \N)>0$. We also fix i.i.d.~samples $X_1, X_2, \dots$ of $\mu$. The proof is divided into two steps.\\

\textbf{Step 1:} Let $(I_k)_{k\in \N}$ be a partition of $\{1,2,\dots, N\}$ such that $\abs{I_k}=L_{k}$ for all $k\in\N$. We write $X_{J}:=\{X_{i} \st i\in J\}$ for any $J \subseteq \N$. In particular, $X_\varnothing = \varnothing$. Let us first prove that
\begin{align}
    \P\left(\mu_N^k\in B_k \text{ for all } k\in \N \st X_{I_k}\subseteq A^k \text{ for all } k\in \N\right)
    =\prod_{k=0}^{\infty}\P\left( (\mu^k)_{L_{k}}\in B_k\right)\,
\end{align}
whenever $\P(X_{I_k}\subseteq A^k \text{ for all } k\in \N)>0$. From independence we obtain
\begin{align}
    &\P\left(\mu_N^k\in B_k \text{ for all } k\in \N \st X_{I_k}\subseteq A^k \text{ for all } k\in \N\right)\\
    &=\P\left(\frac{1}{\abs{I_k}}\sum_{j\in I_k}\delta_{X_{j}}\in B_k \text{ for all } k\in \N \text{ such that } I_k\neq \varnothing \st X_{I_k}\subseteq A^k \text{ for all } k\in \N\right)\\
    &\quad \times \P\left(\delta_0\in B_k \text{ for all } k\in \N \text{ such that } I_k=\varnothing\right)\\
    &=\prod_{k\in\N : I_k\neq \varnothing}
    \P\left(\frac{1}{\abs{I_k}}\sum_{j\in I_k}\delta_{X_{j}}\in B_k \st  X_{I_k}\subseteq A^k\right)
    \times \prod_{k\in\N : I_k=\varnothing}
    \P\left(\delta_0\in B_k \right).
\end{align}
If $I_k\neq \varnothing$, then under $\P_k(\cdot):=\P(\cdot  \st X_{I_k}\subseteq A^k )$, $\{X_{i} \st i\in I_k\}$ are i.i.d.~samples of $\mu^k$, and the sample size is $L_{k}=\abs{I_k}$. Recall that if $I_k=\varnothing$, the empirical measure $(\mu^k)_{L_k}$ has sample size $0$ and is therefore equal to $\delta_0$ by convention. Thus, we have proven the desired result.\\

\textbf{Step 2:} Let us write
\begin{align}
    &\mathcal{J}:=\{ (I_k)_{k\in \N} \st (I_k)_{k\in \N} \text{ is a partition of } \{1,2,\dots, N\} \text{ such that } \abs{I_k}=L_{k} \text{ for all } k\in \N\},\\
    &\mathcal{J}_0:=\{(I_k)_{k\in \N} \in \mathcal{J} \st \P(X_{I_k}\subseteq A^k \text{ for all } k\in \N)>0\}.
\end{align}
Note that
\begin{align}
    \left\{N\mu_N(A^k)=L_{k} \text{ for all } k\in \N\right\}
    =\bigcup_{(I_k)_{k\in \N}\in \mathcal{J}}\left\{ X_{I_k}\subseteq A^k \text{ for all } k\in \N\right\}\,
\end{align}
and the union is disjoint. In particular,
\begin{align}
    \P(N\mu_N(A^k)=L_{k} \text{ for all } k\in\N)
    =\sum_{(I_{k})_{k\in \N}\in \mathcal{J}_0}\P(X_{I_k}\subseteq A^k \text{ for all } k\in \N).
\end{align}
Thus it follows from the previous step that
\begin{align}
    &\P\left(\mu^k_N\in B_k \text{ for all }k\in \N \st N\mu_N(A^k)=L_{k} \text{ for all } k\in \N\right)\\
    &=\sum_{(I_k)_{k\in \N}\in \mathcal{J}_0}\P\left(\mu_N^k\in B_k \text{ for all } k\in \N \st X_{I_k}\subseteq A^k \text{ for all } k\in \N\right)\frac{\P(X_{I_k}\subseteq A^k \text{ for all } k\in \N)}{\P(N\mu_N(A^k)=L_{k} \text{ for all } k\in \N)}\\
    &=\prod_{k=0}^{\infty}\P\left( (\mu^k)_{L_{k}}\in B_k\right)
    \frac{\sum_{(I_k)_{k\in \N}\in \mathcal{J}_0}\P(X_{I_k}\subseteq A^k \text{ for all } k\in \N)}{\P(N\mu_N(A^k)=L_{k} \text{ for all } k\in \N)}\\
    &=\prod_{k=0}^{\infty}\P\left( (\mu^k)_{L_{k}}\in B_k\right).
\end{align}
This ends the proof.
\end{proof}

In the next two lemmas, we apply the local estimates of Lemma~\ref{lem:cmpbd} to estimate uniform deviation probabilities that will appear in the proof of Theorem~\ref{thm:goal}. 

\begin{lem}\label{lem:cpt union bd} Suppose Assumption~\ref{ass:cst} is satisfied and fix $\delta>0$. Then the following hold:
\begin{anumerate}
    \item\label{lem:cpt union bd a} There exist positive constants $c$ and $C$ such that for all $k,N\in \N$ and $x>0$,
    \begin{align}
        \P\left(\frac{2^{-k\delta}}{g(2^{k+1})}\big(\mu(A^k)\wedge \mu_N(A^k)\big)^{\eta}\d_f(\mu^k, \mu^k_N) > x \st \g_N\right)
        \le Ce^{-c4^{k\delta}N\varphi_{\eta}(x)}\ind{\{x\le 1\}} \,\,\as 
    \end{align}
    where $c$ and $C$ depend only on $d,p$. Here $0/0$ is understood as $0$.
    \item\label{lem:cpt union bd b} There exist positive constants $c$ and $C$ such that for all $N\in \N$ and $x>0$,
    \begin{align}
        \P\left(\sup_{k\in \N}\left(\frac{2^{-k\delta}}{g(2^{k+1})}\big(\mu(A^k)\wedge \mu_N(A^k)\big)^{\eta}\d_f(\mu^k, \mu^k_N)\right) > x \right)
    \le Ce^{-c N \varphi_{\eta}(x)}\ind{\{x\le 1\}},
    \end{align}
    where $c$ depends only on $d,p$ and $C$ depends only on $d,p,\delta$. Here $0/0$ is understood as $0$.
\end{anumerate}
\end{lem}
\begin{proof}
$\ref{lem:cpt union bd a}$: If $\mu(A^k)=0$ or $\mu_N(A^k)=0$, the probability in the statement vanishes. Thus, we may assume that $\mu(A^k)> 0$ and $\mu_N(A^k)> 0$ throughout the proof. We claim that it suffices to prove that
\begin{align}\label{eq:cpt union bd clm}
    \P\left(\frac{2^{-k\delta}}{g(2^{k+1})}\d_f(\mu^k, \mu^k_N) > y \st \g_N\right)
    \le Ce^{-c4^{k\delta}N\mu_N(A^k)\varphi(y)}\ind{\{y\le 1\}} \,\,\as
\end{align}
Once \eqref{eq:cpt union bd clm} is proven, set $x=y\big(\mu(A^k)\wedge \mu_N(A^k)\big)^{\eta}$. Then \eqref{eq:cpt union bd clm} gives 
\begin{align}
    &\P\left(\frac{2^{-k\delta}}{g(2^{k+1})}\big(\mu(A^k)\wedge \mu_N(A^k)\big)^{\eta}\d_f(\mu^k, \mu^k_N) > x \st \g_N\right)\\
    &\le C\exp\left(-c 4^{k\delta} N\mu_N(A^k)\varphi\left(\frac{x}{\big(\mu(A^k)\wedge \mu_N(A^k)\big)^{\eta}}\right)\right)\ind{\{x\le (\mu(A^k)\wedge \mu_N(A^k))^{\eta}\}} \,\,\as
\end{align}
Using Lemma~\ref{lem:phi}$\ref{lem:phi b}$ and $\mu_N(A^k)\le 1$, we deduce that
\begin{align}
    &C\exp\left(-c 4^{k\delta} N\mu_N(A^k)\varphi\left(\frac{x}{\big(\mu(A^k)\wedge \mu_N(A^k)\big)^{\eta}}\right)\right)\ind{\{x\le (\mu(A^k)\wedge \mu_N(A^k))^{\eta}\}}\\
    &\le C\exp\left(-c 4^{k\delta} N\mu_N(A^k)\varphi\left(\frac{x}{( \mu_N(A^k))^{\eta}}\right)\right)\ind{\{x\le ( \mu_N(A^k))^{\eta}\}}\\
    &\le C\exp\left(-c 4^{k\delta} N\varphi_{\eta}(x)\right)\ind{\{x\le ( \mu_N(A^k))^{\eta}\}} \\
    &\le C\exp\left(-c 4^{k\delta} N\varphi_{\eta}(x)\right)\ind{\{x\le 1\}}.
\end{align}
Hence it is enough to prove \eqref{eq:cpt union bd clm}.

Now, let us prove \eqref{eq:cpt union bd clm}. From $\mu(A^k)>0$ and $\mu_N(A^k)>0$, $\mu^k$ and $\mu^k_N$ are supported on $A^k$. Thus $T^k_{\#}\mu^k$ and $T^k_{\#}\mu^k_N$ are supported on $A^0$. Recalling the (local) growth function $g$ from Assumption~\ref{ass:cst},
\begin{equation}~\label{eq:d,g}
\begin{aligned}
    \d_f(\mu^k, \mu^k_N)
    &=\inf_{\pi\in \text{Cpl}(\mu^k, \mu^k_N)}\int_{\R^d\times \R^d}f\left(2^{k}\abs{x-y}\right)\,(T^k, T^k)_{\#}\pi(dx,dy)\\
    &=\inf_{\pi\in \text{Cpl}(T^k_{\#}\mu^k, T^k_{\#}\mu^k_N)}\int_{\R^d\times \R^d}f(2^k\abs{x-y})\,\pi(dx, dy) \\
    &\le 2^{-p}g(2^{k+1})\t_p(T^k_{\#}\mu^k, T^k_{\#}\mu^k_N).
\end{aligned}
\end{equation}
Here, the inequality follows from $f(2^k\abs{x-y})=f(2^{k+1}\frac{\abs{x-y}}{2})\le 2^{-p}g(2^{k+1})\abs{x-y}^p$ if $x,y\in A^0$. The estimate \eqref{eq:d,g} combined with Lemmas~\ref{lem:cmpbd} and~\ref{lem:cond} yields that
\begin{equation} \label{eq:cpt union bd 1}
\begin{aligned}
    &\P\left(\frac{2^{-k\delta}}{g(2^{k+1})}\d_f(\mu^k, \mu^k_N)>y \st \g_N\right)\\
    &\le \P\bigg(\t_p(T^k_{\#}\mu^k, T^k_{\#}\mu^k_N)>2^{p+k\delta}y \st \g_N\bigg)\\
    &=\sum_{(L_{\ell})_{\ell\in \N}\in \mathbf{L}_N}\P^{(L_{\ell})}(\t_p(T^k_{\#}\mu^k, T^k_{\#}\mu^k_N)>2^{p+k\delta}y)\ind{\{N\mu_N(A^{\ell})=L_{\ell} \text{ for all } \ell\in \N\}}\\
    &=\sum_{(L_{\ell})_{\ell\in \N}\in \mathbf{L}_N}\P(\t_p(T^k_{\#}\mu^k, T^k_{\#}(\mu^k)_{L_{k}})>2^{p+k\delta}y)\ind{\{N\mu_N(A^{\ell})=L_{\ell} \text{ for all } \ell\in \N\}}\\
    &\le \sum_{(L_{\ell})_{\ell\in \N}\in \mathbf{L}_N}C e^{-cL_{k}\varphi(2^{p+k\delta} y)}\ind{\{N\mu_N(A^{\ell})=L_{\ell, N} \text{ for all } \ell\in \N\}} \\
    &=Ce^{-c N\mu_N(A^k)\varphi(2^{p+k\delta}y)}\,
\end{aligned}
\end{equation}
for some positive constants $c$ and $C$ that depend only on $d,p$. In the second inequality, we use the fact that $T^k_{\#}(\mu^k)_{L_k}$ has the same distribution as the empirical measure of $T^k_{\#}\mu^k$ with sample size $L_k$. From Lemma~\ref{lem:phi}$\ref{lem:phi a}$, $\varphi(2^{p+k\delta}y)\ge 4^p 4^{k\delta}\varphi(y)$. Plugging this back into \eqref{eq:cpt union bd 1}, we get the desired estimate where $4^p$ is absorbed into the constant $c$. The fact that $y>1$ gives zero probability is because $\d_f(\mu^k, \mu^k_N)\le g(2^{k+1})$ which follows from \eqref{eq:d,g} and $\t_p(T^k_{\#}\mu^k, T^k_{\#}\mu^k_N)\le 2^p$.\\

$\ref{lem:cpt union bd b}$: Using $4^{k\delta}\ge 1+k\log(4^{\delta})$, we can apply Lemma~\ref{lem:compute bd} to $\ref{lem:cpt union bd a}$ by setting $A=C$, $a=4^{\delta}$, $\beta=1$ and $h(x)=c\varphi_{\eta}(x)$. This completes the proof.
\end{proof}

\begin{lem}\label{lem:cpt union bd2} Suppose that Assumption~\ref{ass:cst} is satisfied and fix $\delta>0$. Then the following hold:
\begin{anumerate}
    \item\label{lem:cpt union bd2 a} For all $k, N\in\N$ and $x>0$,
    \begin{align}
        \P\left(\frac{2^{-k\delta}}{2G(2^k)}\big(\mu(A^k)\wedge \mu_N(A^k)\big)^{\eta}\abs{M_1(\mu^k; G)-M_1(\mu^k_N; G)}\ge x \st \g_N\right)
        \le 2e^{-2^{1+2k\delta}N\varphi_{\eta}(x)}\ind{\{x\le 1\}}\,\,\as
    \end{align}
    Here $0/0$ is understood as $0$.
    \item\label{lem:cpt union bd2 b}  For all $N\in \N$ and $x>0$,
    \begin{align}
        \P\left(\sup_{k\in \N}\left(\frac{2^{-k\delta}}{2G(2^k)}\big(\mu(A^k)\wedge \mu_N(A^k)\big)^{\eta}\abs{M_1(\mu^k; G)-M_1(\mu^k_N; G)}\right)\ge x\right)
    \le Ce^{-2N\varphi_{\eta}(x)}\ind{\{x\le 1\}},
    \end{align}
    where $C=e\vee\frac{2}{1-4^{-\delta}}$. Here $0/0$ is understood as $0$.
\end{anumerate}
\end{lem}
\begin{proof} $\ref{lem:cpt union bd2 a}$: As in the proof of Lemma~\ref{lem:cpt union bd}, we may assume that $\mu(A^k)>0$ and $\mu_N(A^k)>0$. Similarly, we obtain the desired result by setting $x=y\big(\mu(A^k)\wedge \mu_N(A^k)\big)^{\eta}$ and applying Lemma~\ref{lem:phi}$\ref{lem:phi b}$ to the following estimate:
\begin{align}\label{eq:cpt union bd2 clm}
    \P\left(\frac{2^{-k\delta}}{2G(2^k)}\abs{M_1(\mu^k; G)-M_1(\mu^k_N; G)}\ge y \st \g_N\right)
    \le 2e^{-2^{1+2k\delta}N\mu_N(A^k)\varphi(y)}\ind{\{y\le 1\}}\,\,\as
\end{align}
Let us prove \eqref{eq:cpt union bd2 clm}. Note that $\mu^k$ and $\mu^k_N$ are compactly supported on $A^k$. From $G(\abs{a})\le G(2^k)$ for $a\in A^k$ we obtain $\abs{M_1(\mu^k; G)-M_1(\mu^k_N; G)}\le 2G(2^k)$. Hence the probability in \eqref{eq:cpt union bd2 clm} vanishes when $G(2^k)=0$ or $y>1$. Now, further assume that $G(2^k)>0$ and $y\le 1$. Rewriting the probability in \eqref{eq:cpt union bd2 clm} using Lemma~\ref{lem:cond} and applying Hoeffding's lemma (see \cite[Theorem 2.1]{devroye2001combinatorial}), we get
\begin{equation} \label{eq:clm2 pf}
\begin{aligned}
    &\P\left(\frac{2^{-k\delta}}{2G(2^k)}\abs{M_1(\mu^k; G)-M_1(\mu^k_N; G)}\ge y \st \g_N\right)\\
    &=\sum_{(L_{\ell})_{\ell\in \N}\in \mathbf{L}_N}\P\left(\frac{2^{-k\delta}}{2G(2^k)}\abs{M_1(\mu^k; G)-M_1((\mu^k)_{L_{k}}; G)}\ge y\right)\ind{\{N\mu_N(A^{\ell})=L_{\ell} \text{ for all } \ell\in \N\}}\\
    &\le \sum_{(L_{\ell})_{\ell\in \N}\in \mathbf{L}_N}2 \exp\left(-2\frac{(2^{1+k\delta}G(2^k)L_{k}y)^2}{L_k (M-m)^2}\right)\ind{\{N\mu_N(A^{\ell})=L_{\ell} \text{ for all } \ell\in \N\}}\,
\end{aligned}
\end{equation}
for constants $m$ and $M$ such that $m \le G(|x|) \le M$ for all $x \in A^k$. Since $(M-m)^2\le 4G(2^k)^2$ and $y^2\ge \varphi(y)$ for $y\le 1$ we have
\begin{align}
    \exp\left(-2\frac{(2^{1+k\delta}G(2^k)L_{k}y)^2}{L_k (M-m)^2}\right)
    \le e^{-2^{1+2k\delta}L_{k}y^2}
    \le e^{-2^{1+2k\delta}L_{k}\varphi(y)}.
\end{align}
Plugging this back into \eqref{eq:clm2 pf}, we get the desired estimate.

$\ref{lem:cpt union bd2 b}$: Using $2^{2k\delta}\ge 1+k\log(4^{\delta})$, we can apply Lemma~\ref{lem:compute bd} to $\ref{lem:cpt union bd2 a}$ with $A=2$, $a=4^{\delta}$, $\beta=1$ and $h(x)=2\varphi_{\eta}(x)$. This completes the proof.
\end{proof}

The following lemma gives an upper bound on the uniform deviation of self-normalized empirical processes. The base case $\delta=0$ and $\alpha=1/2$ is studied in \cite{vapnik1974theory, anthony1993result, bartlett1999inequality} and \cite[Exercise 3.3, Exercise 3.4]{devroye2001combinatorial}. In these works, the upper bound depends on the shatter coefficient of the sets $\{A_k\}_{k\in \N}$. We improve this estimate by introducing a factor of $2^{-k \delta}$ for some $\delta>0$ in the statement below. 

\begin{lem}[Uniform deviation of self-normalized empirical processes]\label{lem:relbd} Let $\mu$ be any probability measure on $\R^d$. Then
\begin{align}
    &\P\left(\sup_{k\in \N}2^{-k \delta}\frac{(\mu(A^k)-\mu_N(A^k))_{+}}{\mu(A^k)^{\alpha}}>x\right)
    \le Ce^{-B_{\alpha}x^2 N^{2(1-\alpha\vee 1/2)}}\ind{\{x\le 1\}},\\
    &\P\left(\sup_{k\in \N}2^{-k \delta}\frac{(\mu_N(A^k)-\mu(A^k))_{+}}{\mu_N(A^k)^{\alpha}}>x\right)
    \le Ce^{-B_{\alpha}x^2 N^{2(1-\alpha\vee 1/2)}}\ind{\{x \le 1\}}\,
\end{align}
holds for all $N\in \N$, $\delta>0$, $\alpha\in (0,1)$ and $x>0$, where $B_{\alpha}:=1/2^{3+2\alpha}$ and $C=8e\vee \frac{16}{1-4^{-\delta}}$. Here $0/0$ is understood as $0$.
\end{lem}

\begin{proof}
As $\alpha<1$, the ratios inside the probabilities are no greater than $1$. Hence, we may assume $x\le 1$ throughout the proof because the probabilities in the statement vanish if $x>1$. Let $X_1, X'_1, \dots, X_N, X'_N$ be i.i.d.~samples of $\mu$. Take $\mu_N=(1/N)\sum_{j=1}^N \delta_{X_j}$ and $\mu'_N=(1/N)\sum_{j=1}^N \delta_{X'_j}$. Using the fact that the function $y\mapsto \frac{(y-c)_{+}}{y^{\alpha}}$ is increasing for $c>0$ we have
\begin{align}
    \left\{2^{-k \delta}\frac{(\mu(A^k)-\mu_N(A^k))_{+}}{\mu(A^k)^{\alpha}}>x, \,\, \mu'_N(A^k)\ge \mu(A^k)\right\}
    \subseteq
    \left\{2^{-k \delta}\frac{(\mu'_N(A^k)-\mu_N(A^k))_{+}}{((\mu_N(A^k)+\mu'_N(A^k))/2)^{\alpha}}>x\right\}.
\end{align}
Taking a union over all $k\in \N$ on both sides, it follows from independence of $(X_1,\dots, X_N)$ and $(X_1', \dots, X_N')$ that
\begin{align}
    &\P\left(\sup_{k\in \N} 2^{-k \delta}\frac{(\mu(A^k)-\mu_N(A^k))_{+}}{\mu(A^k)^{\alpha}}>x\right)\P\left(\mu'_N(A^k)\ge \mu(A^k) \text{ for some } k\in \N\right)\\
    &\le \P\left(\sup_{k\in \N}2^{-k \delta}\frac{(\mu'_N(A^k)-\mu_N(A^k))_{+}}{((\mu_N(A^k)+\mu'_N(A^k))/2)^{\alpha}}>x\right).
\end{align}
Since $\mu'_N(A^k)$ follows a binomial distribution $\text{Bin}(N, \mu(A^k))$, $\P(\mu'_N(A^k))\ge \mu(A^k))\ge 1/4$. In particular,
\begin{align}
    \P\left(\sup_{k\in \N} 2^{-k \delta}\frac{(\mu(A^k)-\mu_N(A^k))_{+}}{\mu(A^k)^{\alpha}}>x\right)
    \le 4\P\left(\sup_{k\in \N}2^{-k \delta}\frac{(\mu'_N(A^k)-\mu_N(A^k))_{+}}{((\mu_N(A^k)+\mu'_N(A^k))/2)^{\alpha}}>x\right).\label{eq:relbd1}
\end{align}
Let $\sigma_1, \dots, \sigma_N$ be i.i.d.~Rademacher variables that are also independent of $X_1,\dots, X_N$ and $X'_1, \dots, X'_N$. By the independence lemma (see \cite[Lemma 2.3.4]{shreve2004stochastic}),
\begin{equation}\label{eq:relbe2}
\begin{aligned}
    &\P\left(\sup_{k\in \N}2^{-k \delta}\frac{(\mu'_N(A^k)-\mu_N(A^k))_{+}}{((\mu_N(A^k)+\mu'_N(A^k))/2)^{\alpha}}>x\right)\\
    &=\P\left(\sup_{k\in \N}2^{-k \delta}\frac{\left(\sum_{j=1}^N \sigma_j(\ind{\{X'_j\in A^k\}}-\ind{\{X_j\in A^k\}})\right)_{+}}{\left(\sum_{j=1}^N (\ind{\{X_j\in A^k\}}+\ind{\{X'_j\in A^k\}})\right)^{\alpha}}>\frac{x N^{1-\alpha}}{2^{\alpha}}\right)\\
    &\le 2\P\left(\sup_{k\in \N}2^{-k \delta}\frac{\abs{\sum_{j=1}^N \sigma_j\ind{\{X_j\in A^k\}}}}{\left(\sum_{j=1}^N (\ind{\{X_j\in A^k\}}+\ind{\{X'_j\in A^k\}})\right)^{\alpha}}>\frac{x N^{1-\alpha}}{2^{1+\alpha}}\right)\\
    &\le 2\P\left(\sup_{k\in \N}2^{-k \delta}\frac{\abs{\sum_{j=1}^N \sigma_j\ind{\{X_j\in A^k\}}}}{\left(\sum_{j=1}^N \ind{\{X_j\in A^k\}}\right)^{\alpha}}>\frac{x N^{1-\alpha}}{2^{1+\alpha}}\right)
    =2\E\left[g(X_1, \cdots, X_N)\right],
\end{aligned}
\end{equation}
where
\begin{align}
    g(x_1, \dots, x_N)
    :=\P\left(\sup_{k\in \N}2^{-k \delta}\frac{\abs{\sum_{j=1}^N \sigma_j\ind{\{x_j\in A^k\}}}}{\left(\sum_{j=1}^N \ind{\{x_j\in A^k\}}\right)^{\alpha}}>\frac{x N^{1-\alpha}}{2^{1+\alpha}}\right).
\end{align}
For each $k\in \N$, let us write $J_{k}:=\{1\le j\le N \st x_{j}\in A^{k}\}$. Then
\begin{align}
    \sup_{k\in \N}2^{-k\delta}\frac{\big|\sum_{j=1}^N \sigma_j\ind{\{x_j\in A^k\}}\big|}{\left(\sum_{j=1}^N \ind{\{x_j\in A^k\}}\right)^{\alpha}}
    =\sup_{k\in \N}\frac{2^{-k\delta}}{\abs{J_k}^{\alpha}}\Big|\sum_{j\in J_{k}} \sigma_j\Big|.
\end{align}
Note that $\sigma_j$ are bounded random variables. Applying Hoeffding's lemma (see \cite[Theorem 2.1]{devroye2001combinatorial}) we obtain
\begin{align}
    \P\left(
    \frac{2^{-k\delta}}{\abs{J_k}^{\alpha}}\Big|\sum_{j\in J_{k}} \sigma_j\Big|>\frac{x N^{1-\alpha}}{2^{1+\alpha}}\right)
    &\le 2\exp{\left({-4^{k\delta}B_{\alpha}x^2 N^{2(1-\alpha)}}\abs{J_k}^{2\alpha-1}\right)}\ind{\{\abs{J_{k}}\neq 0\}}\, \label{eq:hl}
\end{align}
for $B_{\alpha}=1/2^{3+2\alpha}$. Note that $\abs{J_k}\neq 0$ implies $\abs{J_k}\ge 1$. Hence if $\alpha\in [1/2, 1)$, $\abs{J_k}^{2\alpha-1}\ge 1$. From $\abs{J_k}\le N$ we deduce that $\abs{J_k}^{2\alpha-1}\ge N^{2\alpha-1}$ if  $\alpha\in (0, 1/2]$. Plugging this estimate back into \eqref{eq:hl},
\begin{align}
    \P\left(
    \frac{2^{-k\delta}}{\abs{J_k}^{\alpha}}\Big| \sum_{j\in J_{k}} \sigma_j \Big|>\frac{x N^{1-\alpha}}{2^{1+\alpha}}\right)
    &\le 2\exp{\left(-4^{k\delta}B_{\alpha}x^2N^{2(1-\alpha\vee 1/2)}\right)}\label{eq:hl2}
\end{align}
follows. Using $4^{k\delta}\ge 1+k\log(4^{\delta})$, the right hand side of \eqref{eq:hl2} is bounded by
\begin{align}
    2\exp{\left(-(1+k\log(4^{\delta})) B_{\alpha}x^2 N^{2(1-\alpha\vee 1/2)})\right)}.
\end{align}
Hence Lemma~\ref{lem:compute bd} gives
\begin{align}
    g(x_1, \cdots, x_N)\le \left(e\vee \frac{2}{1-4^{-\delta}}\right)\exp{\left(-B_{\alpha}x^2N^{2(1-\alpha\vee 1/2)}\right)}.
\end{align}
Combined with \eqref{eq:relbd1} and \eqref{eq:relbe2}, this proves the first estimate in the statement. The proof of the second estimate is essentially the same after changing the roles of $\mu$ and $\mu_N$, see \cite[Proof of Theorem 1]{bartlett1999inequality}.
\end{proof}

We now have all the ingredients needed for the proof of Theorem~\ref{thm:goal}.

\begin{proof}[Proof of Theorem \ref{thm:goal}]
Similarly to \cite[Lemma 5]{fournier2015rate} we write
\begin{align}
    \mu&=\sum_{k=0}^{\infty}(\mu(A^k)\wedge \mu_N(A^k))\mu^k
    +(\mu(A^k)-\mu_N(A^k))_{+}\mu^k,\\
    \mu_N&=\sum_{k=0}^{\infty}(\mu(A^k)\wedge \mu_N(A^k))\mu_N^k
    +(\mu_N(A^k)-\mu(A^k))_{+}\mu^k_N.
\end{align}
Let $\pi^k$ be an optimal transport plan for $\d_f(\mu^k,\mu^k_N)$. Let us define $\pi\in \sp(\R^d\times \R^d)$ via
\begin{align}
    \pi(dx, dy)
    :=\sum_{k=0}^{\infty}(\mu(A^k)\wedge \mu_N(A^k))\pi^k(dx,dy)
    +\frac{\nu(dx)\otimes \lambda(dy)}{\rho},
\end{align}
where 
\begin{align*}
\nu &:=\sum_{k=0}^{\infty}(\mu(A^k)-\mu_N(A^k))_{+}\mu^k,\\
\lambda&:=\sum_{k=0}^{\infty}(\mu_N(A^k)-\mu(A^k))_{+}\mu^k_N,\\
\rho&:=\sum_{k=0}^{\infty}(\mu(A^k)-\mu_N(A^k))_{+}=\sum_{k=0}^{\infty}(\mu_N(A^k)-\mu(A^k))_{+}.
\end{align*}
Clearly, $\pi\in \text{Cpl}(\mu, \mu_N)$. Recalling the (global) growth function $G$ from Assumption~\ref{ass:cst}$\ref{ass:loc g}$,
\begin{equation}\label{eq:bdd}
\begin{aligned}
    \d_f(\mu, \mu_N)
    &\le \int_{\R^d\times \R^d} f(\abs{x-y})d\pi(x,y)\\
    &=\sum_{k=0}^{\infty}(\mu(A^k)\wedge \mu_N(A^k))\d_f(\mu^k, \mu_N^k)
    +\frac{1}{\rho}\int_{\R^d\times \R^d}f(\abs{x-y})d\nu(x)d\lambda(y)\\
    &\le \sum_{k=0}^{\infty}(\mu(A^k)\wedge \mu_N(A^k))\d_f(\mu^k, \mu_N^k)
    +\int_{\R^d}G(\abs{x})d\nu(x)
    +\int_{\R^d}G(\abs{x})d\lambda(x)\,
\end{aligned}
\end{equation}
follows. Next we use the identity $(a-b)_{+}=a-(a\wedge b)$ to obtain
\begin{align}
    \int_{\R^d}G(\abs{x})d\nu(x)
    &=\sum_{k=0}^{\infty}\mu(A^k)M_1(\mu^k ; G)
    -\sum_{k=0}^{\infty}(\mu(A^k)\wedge \mu_N(A^k))M_1(\mu^k ; G)\\
    &=M_1(\mu ; G)-\sum_{k=0}^{\infty}(\mu(A^k)\wedge \mu_N(A^k))M_1(\mu^k ; G).
\end{align}
Similarly we obtain
\begin{align}
    \int_{\R^d}G(\abs{x})d\lambda(x)
    =M_1(\mu_N ; G)-\sum_{k=0}^{\infty}(\mu(A^k)\wedge \mu_N(A^k))M_1(\mu_N^k ; G).
\end{align}
In particular,
\begin{align}
    \int_{\R^d}G(\abs{x})d\lambda(x)
    =\int_{\R^d}G(\abs{x})d\nu(x)
    +\left(M_1(\mu_N ; G)-M_1(\mu ; G)\right)\\
    +\sum_{k=0}^{\infty}(\mu(A^k)\wedge \mu_N(A^k))\left(M_1(\mu^k ; G)-M_1(\mu_N^k ; G)\right).
\end{align}
Plugging this back into \eqref{eq:bdd} we have
\begin{equation}\label{eq:bdd2}
\begin{aligned}
    &\d_f(\mu, \mu_N)\\
    &\le \sum_{k=0}^{\infty}(\mu(A^k)\wedge \mu_N(A^k))\d_f(\mu^k, \mu_N^k)
    +\sum_{k=0}^{\infty}(\mu(A^k)\wedge \mu_N(A^k))\left(M_1(\mu^k ; G)-M_1(\mu_N^k ; G)\right)\\
    & \quad +2\int_{\R^d}G(\abs{x})d\nu(x)
    +\left(M_1(\mu_N ; G)-M_1(\mu ; G)\right) \\
    &=:\text{I}+\text{II}+2\, \text{III}+\text{IV}.
\end{aligned}
\end{equation}
Recall $c_0>0$ from Assumption~\ref{ass:cst}. Note that on the set $\Omega_{\text{I}}$ which is defined via
\begin{align}
    \Omega_{\text{I}}
    :=\left\{\sup_{k\in \N}\left(\frac{2^{-kc_0}}{g(2^{k+1})}\big(\mu(A^k)\wedge \mu_N(A^k)\big)^{\eta}\d_f(\mu^k, \mu^k_N)\right)\le x\right\}\,
\end{align}
we can use the bound
\begin{align}
    \text{I}\le x\sum_{k=0}^{\infty}2^{k c_0}g(2^{k+1})(\mu(A^k)\wedge \mu_N(A^k))^{1-\eta}.
\end{align}
Recalling $\mathcal{K}_g$ and $S$ from Assumption~\ref{ass:cst} and using $M_1(\mu; S)=\int_{\R^d}S(\abs{x})d\mu(x)\ge S(2^{k-1})\mu(A^k)$ for $k\ge 1$, we obtain 
\begin{align}
    \text{I}
    &\le x\sum_{k=0}^{\infty}2^{k c_0}g(2^{k+1})(\mu(A^k)\wedge \mu_N(A^k))^{1-\eta}\\
    &\le x\sum_{k=0}^{\infty}2^{k c_0}g(2^{k+1})(\mu(A^k))^{1-\eta}
    \le x\left(1\vee M_1(\mu; S)\right)^{1-\eta}\mathcal{K}_g\,
\end{align}
for $k\ge 1$.
We bound $\text{II}$ in the similar way. Indeed, on the set
\begin{align}
    \Omega_{\text{II}}
    :=\left\{\sup_{k\in \N}\left(\frac{2^{-kc_0}}{2G(2^k)}(\mu(A^k)\wedge \mu_N(A^k))^{\eta}\abs{M_1(\mu^k ; G)-M_1(\mu^k_N ; G)}\right)\le x\right\}\,
\end{align}
we compute using $\mathcal{K}_G$ from Assumption~\ref{ass:cst} that
\begin{align}
    \text{II}
    &\le 2x\sum_{k=0}^{\infty}2^{kc_0}G(2^k)(\mu(A^k)\wedge\mu_N(A^k))^{1-\eta}\\
    &\le 2x\sum_{k=0}^{\infty}2^{kc_0}G(2^k)(\mu(A^k))^{1-\eta}
    \le 2x(1\vee M_1(\mu; S))^{1-\eta}\mathcal{K}_G.
\end{align}
Finally, we estimate $\text{III}$. For small $\delta>0$ that will be chosen later and for $\ep\in (0, 1-1/\gamma)$, let us define
\begin{align}
    \Omega_{\text{III}}:=\left\{\sup_{k\in \N}2^{-k\delta}\frac{(\mu(A^k)-\mu_N(A^k))_{+}}{(\mu(A^k))^{1/\gamma+\ep}}\le x\right\}.
\end{align}
On the set $\Omega_{\text{III}}$, we have
\begin{align}
    \int_{\R^d}G(\abs{x})d\nu(x)
    &=\sum_{k=0}^{\infty}(\mu(A^k)-\mu_N(A^k))_{+}M_1(\mu^k; G)\\
    &\le x\sum_{k=0}^{\infty}2^{k\delta}(\mu(A^k))^{1/\gamma+\ep}M_1(\mu^k; G) \\
    &=x\sum_{k=0}^{\infty}2^{k\delta}\mu(A^k)^{\ep}\cdot \mu(A^k)^{1/\gamma}M_1(\mu^k; G).
\end{align}
Applying H\"older's inequality  with an exponent $1/\gamma'+1/\gamma=1$ to the last series and Jensen's inequality to obtain $(M_1(\mu^k; G))^{\gamma}\le M_{\gamma}(\mu^k; G)$, we conclude that
\begin{align}
    \int_{\R^d}G(\abs{x})d\nu(x)
    &\le x\left(\sum_{k=0}^{\infty}2^{k\delta\gamma/(\gamma-1)}\mu(A^k)^{\gamma\ep/(\gamma-1)}\right)^{1-1/\gamma}\left(\sum_{k=0}^{\infty} \mu(A^k)(M_1(\mu^k; G))^{\gamma}\right)^{1/\gamma}\\
    &\le x\left(\sum_{k=0}^{\infty}2^{k\delta\gamma/(\gamma-1)}\mu(A^k)^{\gamma\ep/(\gamma-1)}\right)^{1-1/\gamma}(M_{\gamma}(\mu; G))^{1/\gamma}.
\end{align}
Let us choose $\delta<p\ep$ so that $p\gamma\ep/(\gamma-1)-\delta\gamma/(\gamma-1)>0$. It follows from Markov's inequality that $\mu(A^k)\le 2^{-p(k-1)}M_{p}(\mu)$ for $k\ge 1$, which yields
\begin{align}
    &\left(\sum_{k=0}^{\infty}2^{k\delta\gamma/(\gamma-1)}\mu(A^k)^{\gamma\ep/(\gamma-1)}\right)^{1-1/\gamma}\\
    &\le \left(1+(M_{p}(\mu))^{\gamma\ep/(\gamma-1)}2^{p\gamma\ep/(\gamma-1)}\sum_{k=1}^{\infty}2^{-k(p\gamma\ep/(\gamma-1)-\delta\gamma/(\gamma-1))}\right)^{1-1/\gamma}\\
    &\le C(1\vee M_{p}(\mu))^{\ep}\,
\end{align}
for some $C>0$ which depends only on $p, \gamma, \ep$. Plugging these estimates back into \eqref{eq:bdd2}, a union bound gives
\begin{align}
    &\P\bigg(\d_f(\mu, \mu_N)>x(1\vee M_1(\mu; S))^{1-\eta}\mathcal{K}_g+2x(1\vee M_1(\mu; S))^{1-\eta}\mathcal{K}_G \\
    &\hspace{5cm}+Cx(1\vee M_p(\mu))^{\ep}(M_{\gamma}(\mu; G))^{1/\gamma}+x\bigg)\\
    &\le \P(\Omega_{I}^c)
    +\P(\Omega_{II}^c)
    +\P(\Omega_{III}^c)
    +\P(M_1(\mu_N; G)-M_1(\mu; G)>x).
\end{align}
Scaling $x$ in $\Omega_{\text{I}}$, $\Omega_{\text{II}}$ and $\Omega_{\text{III}}$ which only affects the constants $c$ and $A_0$, Lemmas~\ref{lem:cpt union bd},~\ref{lem:cpt union bd2} and~\ref{lem:relbd} together yield
\begin{align}
    \P\left(\d_f(\mu, \mu_N)>Fx\right)
    &\le C\left(e^{-cN\varphi_{\eta}(x)}+Ce^{-cN^{2(1-(\ep+1/\gamma)\vee(1/2))}x^2}\right)\ind{\{x\le A_0\}}\\
    &+\P\left(M_1(\mu_N ; G)-M_1(\mu ; G)>x\right),
\end{align}
where
\begin{align}
    F:=(1\vee M_1(\mu; S))^{1-\eta}+(1\vee M_p(\mu))^{\ep}(M_{\gamma}(\mu; G))^{1/\gamma}.
\end{align}
If $\gamma>2$, we choose $\ep=1/2-1/\gamma$. We can replace $F$ by $F_N$ after changing the roles of $\mu$ and $\mu_N$. This proves the first part of Theorem. If $\gamma\le 2$, then we have $\ep+1/\gamma> 1/2$. Thus the second part of Theorem follows. The dependence of constants $c, C$ and $A_0$ on the parameters stated in the theorem can be easily checked from Lemmas~\ref{lem:cpt union bd},~\ref{lem:cpt union bd2} and~\ref{lem:relbd} together with Lemma~\ref{lem:phi}$\ref{lem:phi c}$. This concludes the proof.
\end{proof}

\noindent\textbf{Acknowledgement}
Johannes Wiesel acknowledges support by NSF Grant DMS-2345556.

\noindent\textbf{Declarations of interest} None.

\begin{small}
\bibliographystyle{abbrv}
\bibliography{Concentration2023}
\end{small}

\end{document}